\documentclass[10pt]{article}
\usepackage[utf8]{inputenc}
\usepackage[a4paper,margin=1.2in]{geometry}

\title{Blow-up for the incompressible 3D-Euler equations with uniform $C^{1,\frac{1}{2}-\epsilon}\cap L^2$ force}
\author{Diego C\'ordoba\footnote{dcg@icmat.es}\quad and Luis Mart\'inez-Zoroa\footnote{luis.martinezzoroa@unibas.ch}}

\usepackage{graphicx}
\usepackage{amsmath}
\usepackage{ dsfont }
\usepackage{amsfonts}
\usepackage{amsthm}
\usepackage{ amssymb }
\usepackage{xcolor}

\newtheorem{theorem}{Theorem}[section]
\newtheorem{corollary}{Corollary}[theorem]
\newtheorem{lemma}[theorem]{Lemma}
\newtheorem{remark}{Remark}
\newtheorem{definition}{Definition}

\def\w{\omega}

\def\p{\partial}

\begin{document}
\maketitle

\begin{abstract}
     This paper presents a novel approach to establish a blow-up mechanism for the forced 3D incompressible Euler equations, with a specific focus on non-axisymmetric solutions. We construct solutions on $\mathbb{R}^3$ within the function space $C^{3,\frac12}\cap L^2$ for the time interval $[0, T)$, where $T > 0$ is finite, subject to a uniform force in $C^{1,\frac12 -\epsilon}\cap L^2$. Remarkably, our methodology results in a blow-up: as the time $t$ approaches the blow-up moment $T$, the integral $\int_0^t |\nabla u| ds$ tends to infinity, all while preserving the solution's smoothness throughout, except at the origin. In the process of our blow-up construction, self-similar coordinates are not utilized and we are able to treat solutions beyond the $C^{1,\frac13+}$ threshold regularity of axy-symmetric solutions without swirl.
    
\end{abstract}

\section{Introduction}

In this paper we consider  the initial value problem for the forced incompressible Euler equations on $\mathbb{R}^3$
\begin{eqnarray}\label{Euler}
\p_t u + (u\cdot\nabla) u + \nabla  p=  f,\\\nonumber
\nabla\cdot u = 0 \\ 
u_0(x)=u(x,0), \nonumber
\end{eqnarray}
 where $u(x,t)=(u_1(x,t), u_2(x,t), u_3(x,t))$ is the velocity field and $p=p(x,t)$ is the pressure function  of an ideal fluid. $f=(f_1(x,t), f_2(x,t), f_3(x,t))$ is an external force. 

 One of the significant unresolved questions in the realms of nonlinear partial differential equations and fluid dynamics is whether smooth initial data  with finite energy can lead to a finite-time singularity/blow-up in the 3D incompressible smooth forced Euler equations (\ref{Euler}).
 
The classical theory of well-posedness for the Euler equation (\ref{Euler}), dating back to the works of Lichtenstein \cite{Lich} and Gunther \cite{Gunther}, asserts the existence of a unique solution in the function space $C^{k,\alpha}\cap L^2$ (with $k\geq1$ and $\alpha\in (0,1)$) for a finite time interval  $[0,T)$. This solution is guaranteed when both the initial divergence free data $u_0(x)$ is in the space $C^{k,\alpha}\cap L^2$ and the external force $f(x,t)$ is uniformly in the space $C^{k,\alpha}\cap L^2$ for all time of existence. Then the natural question is: Do those solutions exist for all time? Or does it exist a finite time T in which there is a singularity?

In the specific context of two-dimensional scenarios, the local well-posedness mentioned above can be seamlessly extended to cover all time intervals (see \cite{Hol,wol}), which can be deduced thanks to the Beale-Kato-Majda criterion \cite{BKM}, as the flow inherently carries the vorticity $\omega = \nabla\times u$ (see also \cite{CFM} for another very interesting geometric criteria for blow-up). However, it's crucial to note that singular behavior remains a possibility. For instance, Kiselev and Šverák \cite{KS} achieved the optimal growth bound for the 2D incompressible Euler equations within a disk. Furthermore, it has been established that the equations exhibit a lack of local well-posedness in critical spaces, notably within $C^k$ spaces, for integers $k\geq 1$. This deficiency due to the non-locality of the pressure was proven by Bourgain and Li \cite{Bourgaincm}, as well as independently by Elgindi and Masmoudi \cite{Elgindi}. Their studies conclusively demonstrated the presence of strong ill-posedness and the absence of uniformly bounded solutions for the initial velocity $u_0$ within the context of $C^k$ (see also \cite{Bourgainsobolev, Elgindisobolev} for similar results in critical Sobolev spaces). Moreover, in a recent collaborative publication \cite{CMZ} in conjunction with Ozanski, we established the existence of a family of global unique classical solutions in the two-dimensional case that display an instantaneous loss of super-critical Sobolev regularity. 

For a more comprehensive historical overview of the finite-time singularity problem in the context of the Euler equation, we recommend consulting the following sources: \cite{BT,Con,Fef,Gib,Hou3,Kis,MB}. These references provide in-depth insights and reviews on the subject matter.

The symmetries inherent in the Euler equations offer a valuable opportunity for investigating axi-symmetric solutions. These solutions are governed by a concise system of two evolution equations that operate within two spatial variables. In particular, when there is no swirl (i.e. absence of angular velocity), it has been shown that for initial data in $C_c^{1,\frac13 +}$ there is global existence (see \cite{Yud} and also the survey \cite{DE} for an interesting review on the well-posedness of axi-symmetric flows), but it turns out that in the case of lower regularity these solutions can develop a singularity. The first proof of finite-time singularities was pioneered by Elgindi in his recent remarkable work \cite{Elgindi2}. In his study, he focused on flows characterized by axi-symmetric symmetry and a lack of swirl, featuring velocity profiles in the $C^{1,\alpha}$ class, where $\frac13>\alpha > 0$ is chosen to be sufficiently small. Elgindi demonstrated the existence of exact self-similar blow-up solutions with infinite energy and no external forces. However, it should be noted that finite-energy blow-up solutions can also be obtained by introducing a uniform non-zero $C^{1,\alpha}$ external force, as discussed in Remark 1.4 of \cite{Elgindi2}. Furthermore, Elgindi, Ghoul, and Masmoudi extended these findings in their work \cite{Elgindi3}, revealing that these blow-up solutions can be localized. Through a detailed analysis of the stability of these blow-up solutions, they established the occurrence of finite-time singularities in the unforced ($f=0$) Euler equation (\ref{Euler}),  for solutions with finite energy in the $C^{1,\alpha}$ space.

Recent advances have brought significant developments in the study of self-similar singularities within axi-symmetric flows when boundaries are involved. We recommend beginning with the work of Elgindi and Jeong \cite{EJ}, where they establish the existence of finite-time blow-up solutions within scale-invariant Holder spaces in domains featuring corners. Furthermore, Chen and Hou, in their work \cite{Hou}, rigorously prove the existence of nearly self-similar $C^{1,\alpha}$ blow-up solutions near smooth boundaries. In their subsequent remarkable research \cite{Hou2, Hou4}, they demonstrate the blow-up of smooth self-similar solutions through the utilization of computer-assisted proofs. Moreover, in \cite{WLGB} by Wang, Lai, Gómez-Serrano, and Buckmaster, they leverage physics-informed neural networks to construct approximate self-similar blow-up profiles.

In collaboration with Zheng, we have recently introduced in \cite{CMZh} a novel mechanism for blow-up in $\mathbb{R}^3$ distinct from the conventional self-similar profile. Specifically, we have devised solutions to the 3D unforced incompressible Euler equations within the domain $\mathbb{R}^3\times [-T,0]$, with velocity profiles belonging to the function space $C^{\infty}(\mathbb{R}^3 \setminus {0})\cap C^{1,\alpha}\cap L^2$, where $0<\alpha\ll\frac13$, for time intervals $t\in (-T,0)$. These solutions are smooth except at a single point and exhibit finite-time singularities precisely at time 0. Notably, these solutions possess axi-symmetric symmetry without swirl, yet they are not founded upon asymptotically self-similar profiles. Instead, they are characterized by an infinite series of vorticity regions, each positioned progressively closer to the origin than its predecessor. The arrangement of vorticity is such that it generates a hyperbolic saddle at the origin, which, as it moves and undergoes deformation, influences the inner vortices. Consequently, the dynamics of the Euler equation within this framework can be approximated by an infinite system of ordinary differential equations  that is explicitly solvable. When we trace the dynamics backward from the blow-up time, this approach empowers us with precise control over the system's behavior, facilitating a rigorous demonstration of the blow-up phenomenon.

The velocity field in \cite{Elgindi2, Elgindi3} exhibits smoothness  away from the axis of symmetry. However, in close proximity to the axis, it possesses only a $C^{1,\alpha}$ regularity. Notably, Chen in \cite{Che2} has recently extended the construction introduced by Elgindi in \cite{Elgindi2}, enabling the application of asymptotically self-similar ansatz to construct blow-up scenarios where the velocity shares the same regularity class as that presented in our prior work in \cite{CMZh}.

The primary objective of this paper is to establish a blow-up mechanism for the forced 3D incompressible Euler equations (\ref{Euler}). Specifically, we aim to provide a non-axisymmetric blow-up mechanism that enables the treatment of smoother solutions and push beyond the $C^{1,\frac13}$ regularity threshold of previous scenarios. In particular, we construct solutions that belong to the function space $C^{3,\frac12}\cap L^2$ over the time interval $[0, T)$ for a finite time $T > 0$, under the influence of a uniform force in $C^{1,\frac12 -\epsilon}\cap L^2$. This construction leads to the property that $$\lim_{t\rightarrow T}\int_0^t |\nabla u (x,s)|_{L^{\infty}} ds = \infty,$$ while the solution remains smooth everywhere except at the origin. We do not employ self-similar coordinates in the construction of our blow-up.

\subsection{The main blow-up result}

The main result of the paper is the following:

\begin{theorem}
For any $\epsilon>0$, there exist  solutions of the forced 3D incompressible Euler equations (\ref{Euler})  in $\mathbb{R}^3\times [0,T]$ with a finite $T>0$ and  with an external forcing which is uniformly in $C^{1,\frac{1}{2}-\epsilon}\cap L^2$, such that on the time interval $0 \le t < T$, the velocity $u$ is in the space $C^{3,\frac12}\cap L^2$ and satisfies $$\lim_{t\rightarrow T}\int_0^t |\nabla u (x,s)|_{L^{\infty}} ds = \infty.$$ 
\end{theorem}
\begin{remark}
    Even though this is not proved in this paper, a careful study of the solution shows that, for any $\delta>0$, the solution for $t\in[0,T-\delta] $ fulfills $u\in C^{\infty}$, and similarly the forcing $f\in C^{\infty}$. Furthermore the vorticity is compactly supported for all $t\in[0,T]$.
\end{remark}
\begin{remark}
The solution at time $T$ is in $C^{1-\delta}$, where $\delta$ depends on $\epsilon$.
\end{remark}
\begin{remark}
    A more careful construction would allow us to show the result with forcing $f$ uniformly bounded in $C^{1,\alpha}$ for all $\alpha\in[0,\frac12)$.
\end{remark}

\subsection{Ideas of the proof}



In order to study our scenario for blow-up, we start by considering $\w(x,t)$ a solution to the 3D-Euler equation in vorticity formulation fulfilling $\w(0,t)=0$, $D^{1}\w(0,t)=0$ and $u(\w(0,t))\approx (x_{1},-x_{2},0)$, where $D^i\w$ refers to the i-th order partial derivatives of $\w$. Then, by adding a small perturbation around the origin $\w_{\lambda}(x)=g(\lambda x)$, formally we get, for very large $\lambda$, the evolution equation

$$\p_{t}\w_{\lambda}(x,t)+(u(\w_{\lambda})+(x_{1},-x_{2},0))\cdot\nabla \w_{\lambda}=\w_{\lambda} \cdot \nabla(u(\w_{\lambda})+(x_{1},-x_{2},0)),$$
$$\w_{\lambda}(x,0)=g(\lambda x).$$
By choosing $g(\lambda x)=(g_{1}(\lambda x),0,g_{3}(\lambda x))$,  this simplifies to
$$\p_{t}\w_{1,\lambda}+(u(\w_{\lambda})+(x_{1},-x_{2},0))\cdot\nabla \w_{1,\lambda}=\w_{\lambda} \cdot \nabla u(\w_{1,\lambda})+\w_{1,\lambda},$$
$$\p_{t}\w_{3,\lambda}+(u(\w_{\lambda})+(x_{1},-x_{2},0))\cdot\nabla \w_{3,\lambda}=\w_{\lambda} \cdot \nabla u(\w_{3,\lambda}),$$
$$\w_{1,\lambda}(x,0)=g_{1}(\lambda x),\w_{3,\lambda}(x,0)=g_{3}(\lambda x).$$

Finally, if we then make the (rather optimistic) assumption that the quadratic terms with respect to $\w_{\lambda}(x,t)$ remain small for some long time, we get the very simple evolution equation

\begin{equation}\label{evolucionheuristica}
    \p_{t}\w_{1,\lambda}+(x_{1},-x_{2},0)\cdot\nabla \w_{1,\lambda}=\w_{1,\lambda},
\end{equation}
$$\p_{t}\w_{3,\lambda}+(x_{1},-x_{2},0)\cdot\nabla \w_{3,\lambda}=0$$
$$\w_{1,\lambda}(x,0)=g_{1}(\lambda x),\w_{3,\lambda}(x,0)=g_{3}(\lambda x).$$

This system of equations can be studied explicitly, and it is relatively straightforward to find initial conditions where $D^{i}u(\w_{1,\lambda})$ grows a lot. The approach we would then like to use to prove our blow-up would be as follows:
\begin{itemize}
    \item We choose $\w(x,t)$ so that its velocity creates a hyperbolic flow around the origin, and with $\w(0,t)=\p_{x_{j}}\w(0,t)=0$,  for $j=1,2,3$. We call $\w(x,t)$ the big scale layer.
    \item We add a small perturbation $\w_{\lambda}(x)$ around the origin, which we call the small scale layer. We consider $\w_{\lambda}(x,t)$ the solution to \eqref{evolucionheuristica}.
    \item We check that $\w(x,t)+\w_{\lambda}(x,t)$ is a solution to the forced incompressible 3D-Euler equation in vorticity formulation. If all the approximations we consider to obtain \eqref{evolucionheuristica} are reasonable, the force necessary to do this should be small and relatively regular.
    \item If we choose $\w_{\lambda}(x,t)$ appropriately, $\w(x,t)$ $\w(x,t)+\w_{\lambda}(x,t)$ can produce a (very strong) hyperbolic flow around the origin for $t\approx 1$, and we can have 
    $$\w(0,t)+\w_{\lambda}(0,t)=\p_{x_{j}}(\w(0,t)+\w_{\lambda}(0,t))=0.$$
    \item We now repeat this gluing process an infinite number of times, using now the small scale layer to make an even more localized layer grow very fast. Using this approach an infinite number of times, we get a blow-up.
\end{itemize}

There are, however, several important complications one needs to take into account when trying to apply these ideas. First, we need to check that \eqref{evolucionheuristica} actually gives us a solution to the forced incompressible 3D-Euler equation, since not all forces are valid (we need, for example, the force to fulfil $\nabla \cdot f =0$).

Even more importantly, we need the quadratic terms with respect to $\w_{\lambda}(x,t)$ i.e.
$$u(\w_{\lambda}(x,t))\cdot\nabla \w_{\lambda}(x,t),\w_{\lambda}(x,t)\cdot\nabla u(\w_{\lambda}(x,t)),$$
to be small. However, in principle, this clashes with our other assumptions that $\w_{\lambda}$ is very localized and that $D^{1}u(\w_{\lambda}(x,t))$ becomes very big. In order to deal with this, we need to find $\w_{\lambda}(x,t)$ where there is some kind of cancellation in the quadratic term, so that it is much smaller than the rough a priori bounds suggest.


Furthermore, we will actually use a more complicated evolution equation than \eqref{evolucionheuristica} in order to obtain a better approximation of the 3D-Euler equation.

\subsection{Outline of the paper}
Section 2 of this paper deals with some notation that we will use through the paper, as well as some important properties of both the velocity operator and the forced incompressible 3D-Euler equations. Section 3 studies the properties of a simplified evolution equation that we will use to model the evolution of the individual layers that will compose our solution to the forced incompressible 3D-Euler equations in vorticity formulation. In section 4, we consider the solutions to the simplified evolution equation and use the properties we have obtained in section 3 to obtain bounds for the different terms that will compose the force in our solutions to the forced incompressible 3D-Euler equations in vorticity formulation, showing in particular that this force has  $C^{\frac12-\epsilon}$ regularity in vorticity formulation (and therefore $C^{1,\frac12-\epsilon}$ in velocity formulation). Finally, in section 5 we use all the previous results to construct the solution that blows up in finite time.

\section{Preliminaries and notation}
\subsection{Some basic notation}
All the norms we will be considering here will refer to spatial norms, that is to say, for example, if $f(x,t)$ is a function defined for $t\in[a,b]$, then
$$||f(x,t)||_{C^{k,\beta}}=\text{sup}_{t_{0}\in[a,b]}||f(x,t_{0})||_{C^{k,\beta}}.$$
We will also use the notation $C^{k+\beta}$ to refer to the $C^{k,\beta}$ spaces.
In general, the domain considered for the time will be clear by context.
We also use the notation
$|f(x,t)|_{C^{k}}$
to refer to the $C^{k}$ seminorm, i.e. if $f(x,t):\mathds{R}^3\times\mathds{R}\rightarrow\mathds{R}$
$$|f(x,t)|_{C^{k}}=\sum_{i=0}^{k}\sum_{j=0}^{k-i}||\frac{\p^{k}f(x,t)}{\p_{x_{1}}^{j}\p_{x_{2}}^{i}\p_{x_{3}}^{k-i-j}}||_{L^{\infty}}.$$
We will use $D^{i}$ to refer to a generic (spatial) derivative of order $i$, and $||D^{i}f||_{C^{j}}$ to refer to the supremum of the $C^{j}$ norm of all the possible $i-$order derivatives of $f$.

\subsection{Properties for the velocity}
We start by noting some properties of the operator $u$ that we will use through the paper. First, we have that, for $\w(x)=(\w_{1}(x),\w_{2}(x),\w_{3}(x))$ with $\nabla\cdot\w(x)=0$ and compactly supported, we can find $u$ fulfilling $\nabla\times u=\w$ by using the Biot-Savart law:

\begin{equation}\label{uwop}
    u(\w)(x)=C\int \frac{(x-y)\times \w(y)dy}{|x-y|^3}dy=C\int \frac{h\times \w(x+h)dh}{|h|^3}dh.
\end{equation}

Since the constant $C$ in the kernel is irrelevant, we will just use from now on $C=1$.

Note that, even though we should only consider $\w$ fulfilling $\nabla\cdot\w=0$, and therefore a vorticity of the form $(\w_{1}(x),0,0)$ would not give us a valid velocity, we can formally define $u(\w_{1}(x),0,0)$ (sometimes written as $u(\w_{1}(x))$ when it is clear by context which component of the vorticity we are considering) as

$$u(\w_{1})(x)=\int \frac{h\times (\w_{1}(x+h),0,0)dh}{|h|^3}dh,$$
and similarly for other components of the vorticity. This will be just a convenient notation. With this notation, the first property we should keep in mind is that, for $i=1,2,3$

$$u_{i}(w_{i})=0$$
which can be readily obtained from \eqref{uwop}. Furthermore, if we consider a derivative of $u_{i}$, we have that $\p_{x_{j}}u_{i}$, $j=1,2,3$, $i=1,2,3$, is a singular integral operator, and in particular we have that, for compactly supported $\w$,
$$||\p_{x_{j}}u_{i}(\w)||_{C^{k,\alpha}}\leq C_{k,\alpha}||\w||_{C^{k,\alpha}}$$
\begin{equation}\label{lnu}
    ||\p_{x_{j}}u_{i}(\w)||_{L^{\infty}}\leq C ||\w||_{L^{\infty}}\ln(10+||\w||_{C^1}).
\end{equation}

Another interesting property of the velocity operator is that it commutes with the rotation operator. If we define, for some  function $f(x)=(f_{1}(x_{1},x_{2},x_{3}),f_{2}(x_{1},x_{2},x_{3}),f_{3}(x_{1},x_{2},x_{3}))$
$$R(f(x))=(f_{3}(x_{2},x_{3},x_{1}),f_{1}(x_{2},x_{3},x_{1}),f_{2}(x_{2},x_{3},x_{1}))$$
we have that
$$R(u(\w))=u(R(\w))$$
which in particular implies that, if $\w(x,t)$ is a solution to the 3D-Euler equation in vorticity formulation, i.e.

$$\p_{t}\w+u(\w)\cdot\nabla\w=\w\cdot\nabla u(\w)$$
then $R(\w)$ is also a solution to the 3D-Euler equation in vorticity formulation. Note that R only affects the spatial variables, and the time variable remains unchanged.

We define similarly $R^{-1}$, the inverse of $R$, which has the same properties regarding the velocity operator and the 3D-Euler equation.
\subsection{Forced 3D-Euler equations}\label{forcedeuler}
Through this paper, we will be studying the forced (incompressible) 3D-Euler equations, i.e.
\begin{equation}\label{velocityequation}
    \p_{t}u+u\cdot\nabla u=-\nabla p+f(x,t)
\end{equation}
with $\nabla \cdot F(x,t)=0$. by taking the curl of this equation, we get
\begin{equation}\label{vorticityequation}
    \p_{t}\w+u\cdot\nabla \w=\w\cdot\nabla u+F(x,t)
\end{equation}
with $F(x,t)=\nabla\times f(x,t).$ We will study \eqref{vorticityequation} in order to obtain information about \eqref{velocityequation}, which can be recovered (if $\w(x,t)$ has enough decay and regularity) via \eqref{uwop}. In order to do so, we need to be a little careful about our choice of $\w$ and $F(x,t)$. First, in order to enssure that $u(x,t)$ is well defined, has finite energy, and that $\w\in C^{\alpha}$ implies $u\in C^{1,\alpha}$ ($\alpha\in(0,1)$), we will only consider $\w(x,t)$ compactly supported, divergence free and with zero average. Similarly, to ensure that $f(x,t)$ is well defined, in $L^2$ and that $F\in C^{\alpha}$ implies $f$ is in $C^{1,\alpha}$ ($\alpha\in(0,1)$), we will only consider $F$ compactly supported, with zero average and divergence free. 

Note that there is a relation between having these properties for $\w$ and for $F$, and in fact we will only show that $\w$ is compactly supported in $B_{1}(0)$ for the times considered, that $\w$ has zero average and that $\nabla \cdot F=\nabla\cdot \w=0$, since this already implies that $F$ has zero average and is supported in $B_{1}(0)$. With this in mind, we give our definition of solution to the forced incompressible 3D-Euler equation.

\begin{definition}\label{defsol}
    We say that $\w(x,t)$ is a solution to the forced incompressible 3D-Euler equation in vorticity formulation with force $F(x,t)$ (or a vorticity solution for short) if for $t\in[a,b]$  $\w(x,t)$ and $F(x,t)$ are supported in some fixed compact $K$, we have $\nabla \cdot F(x,t)=\nabla\cdot \w(x,t)=0$, $\int \w_{i}(x,t)=0$ for $i=1,2,3$, and
    $$\p_{t}\w+u\cdot\nabla \w=\w\cdot\nabla u+F(x,t).$$
    
\end{definition}

\begin{remark}
    The requirements for $\w(x,t)$ to define a vorticity solution can be (significantly) relaxed, but since all the solutions we will consider in this paper fulfil the properties in definition \ref{defsol}, we will use it for simplicity.
\end{remark}

\begin{remark}
    If $\w(x,t)$ is a vorticity solution with force $F(x,t)$, and $u(x,t)$ is the solution to the forced incompressible 3D-Euler equation in velocity formulation that we obtain from $\w(x,t)$, with forcing $f(x,t)$, then $F(x,t)\in C^{\alpha}$ implies $f(x,t)\in C^{1,\alpha}$ for $\alpha\in(0,1).$
\end{remark}

To ensure the condition that $\nabla \cdot F=0$, we will consider only solutions of the form

$$\w(x,t)=\sum_{i=1}^{K}\w^{i}(x,t)$$
with each $\w^{i}(x,t)$ fulfilling
\begin{equation}\label{evolutionlayers}
   \p_{t}\w^{i}(x,t)+u^{i}(x,t)\cdot\nabla \w^{i}(x,t)=\w^{i}(x,t)\cdot \nabla u^{i}(x,t) 
\end{equation}

for some $u^{i}(x,t)$.
If $\nabla\cdot\w^{i}=\nabla\cdot u^{i}=0$ we have that

\begin{align*}
    &\nabla \cdot \p_{t}\w^{i}=\nabla\cdot (u^{i}\cdot\nabla \w^{i}-\w^{i}\cdot \nabla u^{i})=\sum_{k=1}^{3}\p_{x_{k}}\sum_{j=1}^{3}(u_{j}^{i}\p_{x_{j}} \w_{k}^{i}-\w_{j}^{i}\p_{x_{j}}u_{k}^{i})\\
    &=\sum_{k=1}^{3}\sum_{j=1}^{3}(\p_{x_{k}}u_{j}^{i}\p_{x_{j}} \w_{k}^{i}-\p_{x_{k}}\w_{j}^{i}\p_{x_{j}}u^{i}_{k})+\sum_{k=1}^{3}\sum_{j=1}^{3}(u_{j}^{i}\p_{x_{k}}\p_{x_{j}} \w_{j}^{i}-\w^{i}\p_{k}\p_{x_{j}}u^{i}_{k})=0.
\end{align*}
This means that we can write

$$\p_{t}\w+u(\w)\cdot\nabla \w=\w\cdot\nabla u(\w)+F$$
with $F:=\p_{t}\w+u(\w)\cdot\nabla \w-\w\cdot\nabla u(\w)$. Since $ \nabla \cdot (u(\w)\cdot\nabla \w-\w\cdot\nabla u(\w))=\nabla \cdot \p_{t}\w=0$, we have that $ \nabla \cdot  F = 0$.

    \section{The simplified evolution equation}
As mentioned before, to construct our solutions we will divide our solution in different layers, each one with a different spatial scale, and each layer will fulfil an evolution equation like \eqref{evolutionlayers}. For this reason, this first section will be devoted to study this kind of equations, specifically when the velocity $u^{i}$ has some useful properties.
    \begin{definition}
        We say that a velocity field $u(x,t)=(u_{1}(x,t), u_{2}(x,t), u_{3}(x,t))$ is an odd velocity if $x_{i}u_{i}(x,t)$ is even with respect to $x_{1},x_{2}$ and $x_{3}$. Furthermore, we say that $\w(x,t)=(\w_{1}(x,t),\w_{2}(x,t),\w_{3}(x,t))$ is an odd vorticity if $x_{i}\w_{i}(x,t)$ is odd with respect to $x_{1},x_{2}$ and $x_{3}$.
    \end{definition}
    \begin{remark}
        With our definition for $u$ and $\w$ odd, we have that, if $u(x,t)$ is an odd velocity field, then $\nabla\times u(x,t)$ defines an odd vorticity. Similarly, if $\w(x,t)$ is an odd vorticity, \eqref{uwop} gives an odd velocity.
    \end{remark}
    \begin{remark}\label{odd}
        If $u(x,t)$ is an odd velocity then if
        $$\frac{\p^{n_{1}+n_{2}+n_{3}}}{\p x_{1}^{n_{1}}\p x_{2}^{n_{2}} \p x_{3}^{n_{3}}}u_{1}(x=0,t)$$
        with $n_{i}\in \mathds{N}$ is well defined, it is zero unless $n_{1}$ is odd and $n_{2}$ and $n_{3}$ are even. Similar properties hold for the other components of the velocity.

        Similarly, if $\w(x,t)$ is an odd vorticity and if
        $$\frac{\p^{n_{1}+n_{2}+n_{3}}}{\p x_{1}^{n_{1}}\p x_{2}^{n_{2}}\p x_{3}^{n_{3}}}\w_{1}(x=0,t)$$
        is well defined, it is zero unless $n_{1}$ is even and $n_{2}$ and $n_{3}$ are odd. Similar properties hold for the other components of the vorticity.
    \end{remark}
    \begin{definition}\label{ubarra}
        Given a velocity field $u(x,t)$ we define $\bar{u}(x,t)=(\bar{u}_{1}(x,t),\bar{u}_{2}(x,t),\bar{u}_{3}(x,t))$ as
        $$\bar{u}_{j}(x,t)=x_{j}\p_{x_{j}}u_{j}(0,x_{2},0), \ \text{for } j=1,3$$
        $$\bar{u}_{2}(x,t)=u_{2}(0,x_{2},0).$$
    \end{definition}
    \begin{remark}
        If  $\nabla \cdot u(x,t)=0$  then $\nabla \cdot \bar{u}(x,t)=0$.
    \end{remark}

\begin{lemma}\label{phix2}
        Let $\frac{1}{100}>\epsilon>0$, $T_{2}\geq T_{1}$, $P>0$ and $u_{lin}(t)$ and $u_{cub}(x,t)$ (both $u_{lin}$ and $u_{cub}$ depending on a parameter $N$) fulfilling  $u_{lin}(t)\in[\frac12 PN^{\epsilon},2PN^{\epsilon}]$, $||\p_{x}\p_{x}\p_{x}u_{cub}(x,t)||_{L^{\infty}}\leq  N^{4}$ and
        $$u_{cub}(0,t)=\partial_{x}u_{cub}(0,t)=\partial_{x}\partial_{x}u_{cub}(0,t)=0$$
        for $t\in[T_{1},T_{2}]$.
        If we define
        $$\partial_{t} \phi(x,t_{0},t)=-u_{lin}(t)\phi(x,t_{0},t)+u_{cub}(\phi(x,t_{0},t),t)$$
        $$\phi(x,t_{0},t_{0})=x$$
        then for $N$ big enough (depending on $T_{2}-T_{1},P$ and $\epsilon$), we have that, if $x\in[-N^{-3},N^{-3}]$ and $T_{1}\leq t_{0}\leq t\leq T_{2}$, then
        $$|\phi(x,t_{0},t)|\leq x$$
        and
        $$|x|e^{\int_{t_0}^{t}(u_{lin}(s)-x^2N^{4})ds}\leq |\phi(x,t_{0},t)|\leq |x|e^{\int_{t_0}^{t}(u_{lin}(s)+x^2N^{4})ds}$$
        \begin{equation}\label{pphi}
            e^{\int_{t_0}^{t}(u_{lin}(s)-x^2N^{4})ds}\leq |\partial_{x}\phi(x,t_{0},t)|\leq e^{\int_{t_0}^{t}(u_{lin}(s)+x^2N^{4})ds}
        \end{equation}
        $$|\p_{x}\p_{x}\phi(x,t_{0},t)|\leq |t-t_{0}|N^{4}|x|.$$

    \end{lemma}

    \begin{proof}
        First, we note that, for $x\in[-N^{-3},N^{-3}]$, $T_{1}\leq t_{0}\leq t\leq T_{2}$, $N$ big, we have that
        $$|xu_{lin}(t)|>|u_{cub}(x,t)|$$
        which implies
        $$\partial_{t}|\phi(x,t_{0},t)|\leq 0$$
        and thus
        $$|\phi(x,t_{0},t)|\leq x.$$

        Furthermore, we then have
        $|u_{cub}(\phi(x,t_{0},t),t)|<  N^{4}|\phi(x,t_{0},t)|^3\leq  N^{4}|\phi(x,t_{0},t)| x^2$
        so, using this bound and integrating in time we get
        $$|x|e^{\int_{t_0}^{t}(u_{lin}(s)-x^2N^{4})ds}\leq |\phi(x,t_{0},t)|\leq |x|e^{\int_{t_0}^{t}(u_{lin}(s)+x^2N^{4})ds}.$$
        To obtain \eqref{pphi} we note that
        $$\partial_{t} \partial_{x}\phi(x,t_{0},t)=-u_{lin}(t)\partial_{x}\phi(x,t_{0},t)+u_{cub}'(\phi(x,t_{0},t),t)(\partial_{x}\phi(x,t_{0},t))$$
        and using that, for $x\in[-N^{-3},N^{-3}]$, $t\in[0,T]$ for $N$ big,
        $$|u_{cub}'(\phi(x,t_{0},t),t)|< N^{4}\phi(x,t_{0},t)^2\leq N^{4}x^2$$
        we can obtain, again after integrating in time
        $$e^{\int_{t_0}^{t}(u_{lin}(s)-x^2N^{4})ds}\leq |\partial_{x}\phi(x,t_{0},t)|\leq e^{\int_{t_0}^{t}(u_{lin}(s)+x^2N^{4})ds}.$$
        Finally, differentiating the evolution equation again we obtain
        $$\partial_{t} \p_{x}\p_{x}\phi(x,t_{0},t)=-u_{lin}(t)\p_{x}\p_{x}\phi(x,t_{0},t)+u_{cub}'(\phi(x,t_{0},t),t)(\p_{x}\partial_{x}\phi(x,t_{0},t))+u_{cub}''(\phi(x,t_{0},t),t)(\p_{x}\phi(x,t_{0},t))^2$$
        and using that in particular $|\p_{x}\phi(x,t_{0},t)|\leq 1$ and
        $-u_{lin}(t)+u_{cub}'(\phi(x,t_{0},t),t)\leq 0$
        and integrating in time we get
        $$|\p_{x}\p_{x}\phi(x,t_{0},t)|\leq |t-t_{0}|N^{4}|x|.$$
        
    \end{proof}
    
    \begin{lemma}\label{evolucionsimp}
        Given $\frac{1}{100}>\epsilon>0$ and $P>0$, if we have $N,M>0$ big enough (depending on $\epsilon$ and $P$)  and an incompressible, odd velocity field $u^{N}(x,t)=(u^{N}_{1}(x,t),u^{N}_{2}(x,t),u^{N}_{3}(x,t))$  fulfilling $||u^{N}(x,t)||_{C^{3.5}}\leq N^4$, with
    \begin{equation}\label{MN1}
        e^{\int_{1-N^{-\frac{\epsilon}{2}}}^{1}\p_{x_{1}}u_{1}(0,s)ds}=M^{\frac12},
    \end{equation}
    \begin{equation}\label{MN2}
        \p_{x_{1}}u^{N}_{1}(0,t)\in[\frac{P}{2}N^{\epsilon},2PN^{\epsilon}]\text{ for }t\in [1-N^{-\frac{\epsilon}{2}},1]
    \end{equation}

        $$ |\partial_{x_{3}}u^{N}_{3}(x,t)|\leq \ln(N)^3\ \text{for }t\in[1-N^{-\frac{\epsilon}{2}},1]$$
    then, if for some $\w_{end}=(\w_{1,end},\w_{2,end}=0,\w_{3,end})$ with $\text{supp}(\w_{i,end})\subset\{|x_{2}|\leq M^{-\frac{1}{2}-\epsilon}\} $ we define the evolution equation
    \begin{equation}\label{ecuacionsimp}
        \p_{t}\w_{i}+\bar{u}^{N}\cdot\nabla \w_{i}=\w_{i}\p_{x_{i}}\bar{u}^{N}_{i}
    \end{equation}
    $$\w_{i}(x,t=1)=\w_{i,end}(x)$$
    and 
    $$\p_{t}\phi(x,t_{0},t)=\bar{u}^{N}(\phi(x,t_{0},t),t)$$
    $$\phi(x,t_{0},t_{0})=x$$
    we have that
    $$\w_{i}(x,t)=\w_{i,end}(\phi(x,t,1))e^{\int_{1}^{t}(\p_{x_{i}}\bar{u}^{N}_{i})(\phi(x,t,s),s)ds}.$$
    Furthermore, if for $1-N^{-\frac{\epsilon}{2}}\leq t\leq 1$ we define
    $$K_{N}(t)=e^{\int_{t}^{1}\p_{x_{1}}\bar{u}^{N}_{1}(0,s)ds}$$
    then we have that
    \begin{equation}\label{controlsoporte}
        \text{supp}(\w(x,t))\subset\{|x_{2}|\leq 2 M^{-\frac{1}{2}-\epsilon}K(t)\},
    \end{equation}
    and for $1-N^{-\frac{\epsilon}{2}}\leq t\leq 1$, $|x_{2}|\leq 4 M^{-\frac{1}{2}-\epsilon}K(t)$,
\begin{equation}\label{phi1}
    \phi_{1}(x,t,1)=x_{1}K_{N}(t)g_{er,1}(x_{2},t)
\end{equation}
\begin{equation}\label{phi3}
    \phi_{3}(x,t,1)=x_{3}g_{er,3}(x_{2},t)
\end{equation}
with
    \begin{equation}\label{controlphii}
        |g_{er,1}(x_{2},t)|,|g_{er,3}(x_{2},t)|\in[\frac{1}{2},2],|\p_{x_{2}}g_{er,1}(x_{2},t)|,|\p_{x_{2}}g_{er,3}(x_{2},t)|\leq 8N^{4}M^{-\frac{1}{2}-\epsilon}K_{N}(t)
    \end{equation}
    \begin{equation}\label{controlphii2}
        |\p_{x_{2}}\p_{x_{2}}g_{er,3}(x_{2},t)|,|\p_{x_{2}}\p_{x_{2}}g_{er,1}(x_{2},t)|\leq 4N^{4}
    \end{equation}
    and
    \begin{equation}\label{controlphi21}
        \frac{1}{2K_{N}(t)}\leq|\p_{x_{2}}\phi_{2}(x,t,1)|\leq \frac{2}{K_{N}(t)}
    \end{equation}
\begin{equation}\label{controlphi22}
    |\p_{x_{2}}\p_{x_{2}}\phi_{2}(x,t,1)|\leq N^4 M^{-\frac{1}{2}-\epsilon}K_{N}(t).
\end{equation}

Finally, also for $1-N^{-\frac{\epsilon}{2}}\leq t\leq 1$, $x_{2}\in \{|x_{2}|\leq 4 M^{-\frac{1}{2}-\epsilon}K(t)\}$, we have
\begin{equation}\label{amplitud1}
    e^{\int_{1}^{t}(\p_{x_{1}}\bar{u}^{N}_{1})(\phi(x,t,s),s)ds}=\frac{a_{1}(x_{2},t)}{K_{N}(t)}
\end{equation}
\begin{equation}\label{amplitud3}
    e^{\int_{1}^{t}(\p_{x_{3}}\bar{u}^{N}_{3})(\phi(x,t,s),s)ds}=a_{3}(x_{2},t)
\end{equation}
with
\begin{equation}\label{controlamplitud}
        \frac{1}{2}\leq|a_{i}(x_{2},t)|\leq 2, |\p_{x_{2}}a_{i}(x_{2},t)|\leq 1,|\p_{x_{2}}\p_{x_{2}}a_{i}(x_{2},t)|\leq 6N^4.
    \end{equation}
    \end{lemma}

\begin{remark}
    For the initial consitions we consider, the family of evolution equations \eqref{ecuacionsimp} is actually equivalent to 
    $$\p_{t}\omega+\bar{u}^{N}\cdot\nabla \w=\w\cdot\nabla\bar{u}^{N}.$$
\end{remark}


\begin{proof}
    First, we note that, if we define
    $$\tilde{\w}_{i}(x,t)=\w(\phi(x,1,t),t)$$ we get
    \begin{equation}
        \p_{t}\tilde{\w}_{i}(x,t)=\tilde{\w}_{i}(x,t)(\p_{x_{i}}\bar{u}^{N}_{i})(\phi(x,1,t),t)
    \end{equation}
    $$\tilde{\w}_{i}(x,t=1)=\w_{i,end}(x)$$
    so we can solve this problem to get
$$\tilde{\w}_{i}(x,t)=\w_{i,end}(x)e^{\int_{1}^{t}(\p_{x_{i}}\bar{u}^{N}_{i})(\phi(x,1,s),s)ds}.$$
Then, undoing the change of variables and using that
$$\phi(\phi(x,t_{1},t_{2}),t_{2},t_{3})=\phi(x,t_{1},t_{3})$$
we get
$$\w_{i}(x,t)=\w_{i,end}(\phi(x,t,1))e^{\int_{1}^{t}(\p_{x_{i}}\bar{u}^{N}_{i})(\phi(x,t,s),s)ds}.$$
Note that this means that
\begin{equation}\label{suppphi}
    \text{supp}(\w_{i}(x,t))\subset\{x:x=\phi(y,1,t), y\in\text{supp}(\w_{i,end})\}.
\end{equation}

Next, we have that, for $t\in[1-N^{-\frac{\epsilon}{2}},1]$,
\begin{align*}
    \p_{t}\phi_{2}(x,1,t)&=\bar{u}^{N}_{2}(\phi_{2}(x_{2},1,t),t)=\phi_{2}(x_{2},1,t)\p_{x_{2}}\bar{u}^{N}_{2}(0,t)+u^{N}_{2,cub}(\phi(x_{2},1,t))\\
    &=-\phi_{2}(x_{2},1,t)(\p_{x_{1}}\bar{u}^{N}_{1}(0,t)+\p_{x_{3}}\bar{u}^{N}_{3}(0,t))+u^{N}_{2,cub}(\phi(x_{2},1,t))
\end{align*}
with $u^{N}_{2,cub}(x_{2})=\bar{u}^{N}_{2}(x_{2})-x_{2}(\p_{x_{2}}\bar{u}^{N}_{2})(0,t)$, and using that $\bar{u}^{N}$ is odd (see remark \ref{odd}) and the bounds for $u^{N}_{3}$ we get
$$|\p_{t}\phi_{2}(x,1,t)|\leq [|(-\p_{x_{1}}\bar{u}^{N}_{1})(0,t)|+\ln(N)^3|]\phi_{2}(x,1,t)|+N^4|\phi_{2}(x,1,t)|^3$$
so, if $|\phi_{2}(x,1,t)|\leq N^{-3}$ for $t\in[t_{0},1]$, $t_{0}\geq 1-N^{-\frac{\epsilon}{2}}$ we get
\begin{equation}\label{controlsupport}
    |\phi_{2}(x,1,t_{0})|\leq |x_{2}|e^{\int_{1}^{t_{0}}((-\p_{x_{1}}\bar{u}^{N}_{1})(0,t)+\ln(N)^3+1)dt}\leq |x_{2}|K_{N}(t)e^{N^{-\frac{\epsilon}{2}}(\ln(N)^3+1)}\leq 2K_{N}(t_{0})|x_{2}|.
\end{equation}
In particular,
 we get that, for any $t_{0} \in[1-N^{-\frac{\epsilon}{2}},1]$
$$\phi(x_{2},1,t)\in [-N^{-3},N^{-3}]\text{ for } t\in[t_{0},1] \Rightarrow\ \phi_{2}(x,1,t_{0})\in [-2|x_{2}|K_{N}(t),2|x_{2}|K_{N}(t)]\text{ for } t\in[t_{0},1] $$ 
and since for $N$ big, $|x_{2}|\leq 4M^{-\epsilon-\frac{1}{2}}$ we have 
$$[-2|x_{2}|K_{N}(t),2|x_{2}|K_{N}(t)]\subset[-4M^{-\epsilon},4M^{-\epsilon}]\subset [-N^{-3},N^{-3}]$$
the continuity of $\phi(x_{2},1,t)$ with respect to $t$ gives that, in fact, for $t_{0}\in[1-N^{-\frac{\epsilon}{2}},1]$ and $|x_{2}|\leq 4M^{-\epsilon-\frac{1}{2}}$ 
\eqref{controlsupport} holds, which, combined with \eqref{suppphi}, gives us that
$$\text{supp}(\w(x,t))\subset\{|x_{2}|\leq 2 M^{-\frac{1}{2}-\epsilon}K_{N}(t)\}.$$

Now, for $1-N^{-\frac{\epsilon}{2}}\leq t_{0}\leq t\leq 1$, $|x_{2}|\leq 4 M^{-\frac{1}{2}-\epsilon}K_{N}(t)$ we can apply lemma \ref{phix2} to obtain
$$e^{\int_{t_0}^{t}(\p_{x_{2}}(\bar{u}^{N}_{2})(0,t)-x_{2}^2N^{4})ds}\leq |\partial_{x_{2}}\phi_{2}(x,t_{0},t)|\leq e^{\int_{t_0}^{t}(\p_{x_{2}}(\bar{u}^{N}_{2})(0,t)+x_{2}^2N^{4})ds}\leq 1$$
$$|\p_{x_{2}}\p_{x_{2}}\phi_{2}(x_{2},t_{0},t)|\leq |t-t_{0}|N^{4}|x_{2}|,$$
which in particular imply \eqref{controlphi21} and \eqref{controlphi22}.
Now, for the bounds of $\phi_{1}$ and $\phi_{3}$ we note that, for $i=1,3$
$$\p_{t}\phi_{i}(x,t,1)=\bar{u}^{N}_{i}(\phi(x,t,1))=\phi_{i}(x,t,1)(\p_{x_{i}}\bar{u}^{N}_{i})((0,\phi_{2}(x_{2},t,1),0),t)$$
$$\phi_{i}(x,1,1)=x_{i},$$
so, if we define
$$p_{N,i}(x_{2},t)=(\p_{x_{i}}\bar{u}^{N}_{i})((0,\phi_{2}(x_{2},t,1),0),t)-(\p_{x_{i}}\bar{u}^{N}_{i})(0,t)$$
we have that for $i=1,3$

$$\phi_{i}(x,t,1)=x_{i}e^{\int_{t}^{1}((\p_{x_{i}}\bar{u}^{N}_{i})(0,s)+p_{N,i}(\phi(x_{2},t,s),s))ds}$$
so
$$\phi_{1}(x,t,1)=x_{1}K_{N}(t)e^{\int_{t}^{1}p_{N,1}(\phi(x_{2},t,s),s)ds}$$
$$\phi_{3}(x,t,1)=x_{3}e^{\int_{t}^{1}((\p_{x_{3}}\bar{u}^{N}_{3})(0,s)+p_{N,3}(\phi(x_{2},t,s),s))ds}$$
which gives \eqref{phi1} and \eqref{phi3}. Next, using that, for $1-N^{-\frac{\epsilon}{2}}\leq t_{0}\leq t\leq 1$, $|x_{2}|\leq 4 M^{-\frac{1}{2}-\epsilon}K_{N}(t)$
$$|\p_{x_{2}}p_{N,i}(\phi(x_{2},t_{0},t),t)|\leq N^{4}|\phi(x_{2},t_{0},t)||\p_{x_{2}}\phi(x_{2},t_{0},t)|\leq N^{4}|x_{2}|$$
we get 
$$|\p_{x_{2}}e^{\int_{t}^{1}p_{N,i}(\phi(x_{2},t,s),s)ds}|\leq e^{\int_{t}^{1}p_{N,i}(\phi(x_{2},t,s),s)ds}N^{4}|x_{2}|\leq e^{\int_{t}^{1}p_{N,i}(\phi(x_{2},t,s),s)ds}N^{4}4M^{-\frac{1}{2}-\epsilon}K_{N}(t),$$
and since, 
$$|p_{N,i}(\phi(x_{2},t_{0},t),t)|\leq N^{4}|\phi(x_{2},t_{0},t)|^2\leq N^{4}|x_{2}|^2, |\p_{x_{3}}u^{N}_{3}(0,t)|\leq \ln(N)^3$$
we get \eqref{controlphii}.
Similarly, using
$$|\p_{x_{2}}\p_{x_{2}}p_{N,i}(\phi(x_{2},t_{0},t),t)|\leq 2N^{4}$$
gives \eqref{controlphii2}.

For \eqref{amplitud1}, \eqref{amplitud3} and \eqref{controlamplitud} we use that
$$a_{1}(x_{2},t)=e^{\int_{1}^{t}((\p_{x_{1}}\bar{u}^{N}_{1})(\phi_{2}(x_{2},t,s),s)-(\p_{x_{1}}\bar{u}^{N}_{1})(0,s))ds}$$
$$a_{3}(x_{2},t)=e^{\int_{1}^{t}(\p_{x_{3}}\bar{u}^{N}_{3})(\phi_{2}(x_{2},t,s),s)}ds$$
and using the bounds for $(\p_{x_{3}}\bar{u}^{N}_{3})(x_{2},t)$ and for $\bar{u}^{N}_{1}(x_{2},t)$ 
to get, for $|x_{2}|\leq 4 M^{-\frac{1}{2}-\epsilon}K_{N}(t)$, $t\in[1-N^{-\frac{\epsilon}{2}},1]$, that 
$$\frac12\leq e^{-N^4(1-t)x^2_{2}}\leq |a_{1}(x_{2},t)|\leq e^{N^4(1-t)x^2_{2}}\leq 2$$
$$\frac{1}{2}\leq e^{-(1-t)\ln(N)}\leq |a_{3}(x_{2},t)|\leq e^{(1-t)\ln(N)}\leq 2$$
and similarly, using the bounds for $\bar{u}^{N}_{i},\phi_{2}$ and their derivatives, we get

$$|\p_{x_{2}}a_{1}(x_{2},t)|,|\p_{x_{2}}a_{3}(x_{2},t)|\leq  4N^4|x_2|\leq 1$$
and for the second derivatives
$$|\p_{x_{2}}\p_{x_{2}}a_{1}(x_{2},t)|,|\p_{x_{2}}\p_{x_{2}}a_{3}(x_{2},t)|\leq  6N^4$$
which finishes the proof.

\end{proof}

\section{Bounds for the small scale layer}

In order to show that the gluing of the small scale layer and the big scale layer only requires a reasonable force (more specifically, almost $C^{1,\frac{1}{2}}$) to produce the desired behaviour, we need to obtain several useful bounds regarding the properties of the small scale layer. Since all the lemmas obtained in this section will require the same set of hypothesis, and to make the statements more compact, we will start this section specifying the assumptions that will hold for the rest of the section, as well as the notation that we will use.

We start by fixing $\frac{1}{100}>\epsilon>0$, $P>0$ and $f(z)$ a $C^{\infty}$ even function with $1\geq f(z)\geq 0$, $f(z)=1$ if $|z|\leq\frac{1}{2}$, $f(z)=0$ if $|z|\geq 1$. Note that this function $f(z)$ has nothing to do with the force appearing in \ref{Euler}.

The constants in the lemmas will depend on the specific choice of $\epsilon,P$ and $f(z)$, but the result will hold independently of the choice. Given $N>1$ and an incompressible, odd velocity field $u^{N}(x,t)=(u^{N}_{1}(x,t),u^{N}_{2}(x,t),u^{N}_{3}(x,t))$  fulfilling $||u^{N}(x,t)||_{C^{3.5}}\leq N^4$, with
    \begin{equation}\label{MN21}
        e^{\int_{1-N^{-\frac{\epsilon}{2}}}^{1}\p_{x_{1}}u^{N}_{1}(0,s)ds}=M^{\frac12},
    \end{equation}
    \begin{equation}\label{MN22}
        \p_{x_{1}}u^{N}_{1}(0,t)\in[\frac{P}{2}N^{\epsilon},2PN^{\epsilon}]\text{ for }t\in [1-N^{-\frac{\epsilon}{2}},1]
    \end{equation}
we define $\w(x,t)=(\w_{1}(x,t),\w_{2}(x,t),\w_{3}(x,t))$ by
\begin{equation}\label{evolucionsmall}
    \p_{t}\omega+\bar{u}^{N}\cdot\nabla \w=\w\cdot\nabla\bar{u}^{N},
\end{equation}
with $\bar{u}^{N}$ as in definition \ref{ubarra}, and
$$\w_{1}(x,t=1)=-M^{\epsilon}f(M^{\frac{1}{2}-\epsilon}x_{1})\sin(Mx_{2})f(M^{\frac{1}{2}+\epsilon}x_{2})\sin(Mx_{3})f(M^{\frac{1}{2}}x_{3})$$
$$\w_{2}(x,t=1)=0,$$
$$\w_{3}(x,t=1)=-\int_{-\infty}^{x_{3}}\p_{x_{1}}\w_{1}(x_{1},x_{2},s,t=1)ds.$$

We note that, for $t\in[1-N^{-\frac{\epsilon}{2}},1]$, $\w_{2}(x,t)=0$ and $\p_{x_{1}}\w_{1}(x,t)=-\p_{x_{3}}\w_{3}(x,t)$. Furthermore, if $N$ is big enough, we can apply lemma \ref{evolucionsimp}, remembering that in lemma \ref{evolucionsimp} we define $K_{N}(t)=e^{\int_{t}^{1}\p_{x_{1}}\bar{u}^{N}_{1}(0,s)ds}$ and defining $b(t)=\frac{\ln(K_{N}(t))}{\ln(M)}$, we can write $\w_{1}(x,t)=\w_{1,b}(x)$
where
\begin{align}\label{w1b}
   &\w_{1,b}(x):=-M^{\epsilon-b}a_{b}(x_{2})f(M^{\frac{1}{2}+b-\epsilon}x_{1}g_{er,1,b}(x_{2}))\sin(Mg_{2,b}(x_{2}))f(M^{\frac{1}{2}+\epsilon}g_{2,b}(x_{2}))\\ \nonumber
   &\times \sin(Mg_{er,3,b}(x_{2})x_{3})f(M^{\frac{1}{2}}g_{er,3,b}(x_{2})x_{3}) 
\end{align}
with $b\in[0,\frac{1}{2}]$, $\text{supp}(\w_{1,b})(x)\subset \{x:|x_{2}|\leq 2 M^{-\frac{1}{2}-\epsilon+b}\}$
and such that, for $|x_{2}|\leq 4 M^{-\frac{1}{2}-\epsilon+b}$
$$\frac{1}{2}\leq|a_{b}(x_{2})|\leq 2, |\p_{x_{2}}a_{b}(x_{2})|\leq 1,|\p_{x_{2}}\p_{x_{2}}a_{b}(x_{2})|\leq M^{\frac{1}{2}}$$
$$|g_{er,1,b}(x_{2})|,|g_{er,3,b}(x_{2})|\in[\frac{1}{2},2],|\p_{x_{2}}g_{er,1,b}(x_{2})|,|\p_{x_{2}}g_{er,3,b}(x_{2})|\leq M^{-\frac{1}{2}+b},$$
$$|\p_{x_{2}}\p_{x_{2}}g_{er,1,b}(x_{2},t)|,|\p_{x_{2}}\p_{x_{2}}g_{er,3,b}(x_{2},t)|\leq M^{\frac{1}{2}}$$
and
$$\frac{1}{M^{b}}\leq|\p_{x_{2}}g_{2,b}(x_{2})|\leq \frac{2}{M^{b}},|\p_{x_{2}}\p_{x_{2}}g_{2,b}(x_{2})|\leq  M^{-\frac{1}{2}+b},$$
for $N$ big, where we used that, for any fixed $\delta>0$, if $N$ is big enough then $M^{\delta}>N$ (using \eqref{MN21} and \eqref{MN22}).
Note that we can use the incompressibility of $\w$ to obtain a similar expression $\w_{3}(x,t)=\w_{3,b}(x)$, and we will also write $\w_{b}=(\w_{1,b},0,\w_{3,b})$.

We can now use these expressions to obtain information about the properties of $\w(x,t)$. Since the equation \eqref{evolucionsmall} should be obtained by applying an almost $C^{1,\frac{1}{2}}$ force to the 3D-Euler equations, we need to show that the "missing terms" (i.e., the terms that appear in the 3D-Euler equation but not in \eqref{evolucionsmall}) are almost $C^{1,\frac{1}{2}}$ (in the velocity formulation, and therefore almost $C^{\frac{1}{2}}$ in vorticity formulation). In particular, we want the quadratic terms $\w\cdot\nabla u(\w),u(\w)\cdot\nabla\w$ to be almost $C^{\frac{1}{2}}$.

For this, we start by obtaining bounds for $u(\w_{1})$. Note also that, since $M$ and $N$ are related by \eqref{MN21} and \eqref{MN22}, it is equivalent to say "for $N$ big enough" or "for $M$ big enough".

\begin{lemma}\label{vaprox}
There exists $C_{0}>0$ such that for $M>0$ big enough if we define
\begin{align*}
    &\tilde{u}_{2}(\w_{1,b})(x)=C_{0}\frac{Mg_{3,er,b}(x_{2})}{(Mg_{3,er,b}(x_{2}))^2+(M\p_{x_{2}}g_{2,b}(x_{2}))^2}M^{\epsilon-b}a_{b}(x_{2})f(M^{\frac{1}{2}+b-\epsilon}x_{1}g_{er,1,b}(x_{2}))\\
    &\times \sin(Mg_{2,b}(x_{2}))f(M^{\frac{1}{2}+\epsilon}g_{2,b}(x_{2})) \cos(Mg_{er,3,b}(x_{2})x_{3})f(M^{\frac{1}{2}}g_{er,3,b}(x_{2})x_{3})\\
\end{align*}
\begin{align*}
    &\tilde{u}_{3}(\w_{1,b})(x)=-C_{0}\frac{M\p_{x_{2}}g_{2,b}(x_{2})}{(Mg_{3,er,b}(x_{2}))^2+(M\p_{x_{2}}g_{2,b}(x_{2}))^2}M^{\epsilon-b}a_{b}(x_{2})f(M^{\frac{1}{2}+b-\epsilon}x_{1}g_{er,1,b}(x_{2}))\\
    &\times \cos(Mg_{2,b}(x_{2}))f(M^{\frac{1}{2}+\epsilon}g_{2,b}(x_{2})) \sin(Mg_{er,3,b}(x_{2})x_{3})f(M^{\frac{1}{2}}g_{er,3,b}(x_{2})x_{3})\\
\end{align*}
we have that, for any $\delta>0$,$|x_{2}|< 4M^{-\frac{1}{2}-\epsilon+b} $ and $j=0,1$

$$||\tilde{u}_{2}(\w_{1,b})-u_{2}((\w_{1,b},0,0))||_{C^{j}}\leq C_{\delta} M^{2\epsilon+j-\frac{3}{2}+\delta}$$
\begin{equation}\label{linftu3}
    ||\tilde{u}_{3}(\w_{1,b})-u_{3}((\w_{1,b},0,0))||_{C^{j}}\leq C_{\delta} M^{2\epsilon+j-\frac{3}{2}+\delta}
\end{equation}
and, for $i=1,3$, $j=0,1$, $|x_{2}|< 4M^{-\frac{1}{2}-\epsilon+b} $,
$$||\p_{x_{i}}\tilde{u}_{2}(\w_{1,b})-\p_{x_{i}}u_{2}((\w_{1,b},0,0))||_{C^{j}}\leq C_{\delta} M^{2\epsilon+j-\frac{1}{2}+\delta}$$
$$||\p_{x_{i}}\tilde{u}_{3}(\w_{1,b})-\p_{x_{i}}u_{3}((\w_{1,b},0,0))||_{C^j}\leq C_{\delta} M^{2\epsilon+j-\frac{1}{2}+\delta}.$$
\end{lemma}

\begin{proof}
    We will only show the inequalities for $\tilde{u}_{3}$, since the ones for $\tilde{u}_{2}$ are completely analogous. We start by studying the function
    $$u_{3}((\w_{1,b},0,0))=-\int_{\mathds{R}^3} \frac{h_{2}}{|h|^3}\w_{1,b}(x+h)dh_{1}dh_{2}dh_{3}.$$
    In order to show \eqref{linftu3} for $j=0$, we will start by showing that we can make several useful approximations to $u_{3}(\w_{1,b})$ that create an error smaller than $C_{\delta} M^{2\epsilon+j-\frac{3}{2}+\delta}$.
    First, we note that we can use integration by parts and the bounds for $g_{er,3,b}(x_{2})$ to show that
\begin{align}\label{h3small}
    &|\int_{\mathds{R}}\frac{f(M^{\frac{1}{2}}g_{er,3,b}(x_{2}+h_{2})(x_{3}+h_{3}))\sin(Mg_{er,3,b}(x_{2}+h_{2})(x_{3}+h_{3})}{|h|^3}dh_{3}|\\\nonumber
    &=|\int_{\mathds{R}}\frac{\p^k}{\p h_{3}^k}\big(\frac{f(M^{\frac{1}{2}}g_{er,3,b}(x_{2}+h_{2})(x_{3}+h_{3}))}{|h|^3}\big)\frac{\sin(Mg_{er,3,b}(x_{2}+h_{2})(x_{3}+h_{3})+k\frac{\pi}{2})}{(Mg_{er,3,b}(x_{2}+h_{2}))^k}dh_{3}|\\
    &\leq \int_{\mathds{R}}C_{k}\sum_{i=0}^{k}\big(\frac{M^{\frac{i}{2}}}{|h|^{3+k-i}}\big)\frac{2^k}{M^k}dh_{3}\leq C_{k}\sum_{i=0}^{k}\big(\frac{M^{\frac{i}{2}}}{|h_{1}^{2}+h_{2}^{2}|^{\frac{2+k-i}{2}}}\big)\frac{2^k}{M^k}\nonumber
\end{align}
and using this plus that, for $M$ big
$$\text{supp}(\w_{1,b}(x))\subset\{x:|x|\leq 1\}$$
to show that, for any $\frac{1}{4}\geq \delta\geq 0$, $k\in\mathds{N}$, $M$ big
\begin{align*}
    &|\int_{|h_{1}^2+h_{2}^2|\geq M^{-1+\delta},h_{3}\in\mathds{R}}\frac{h_{2}}{|h|^3}\w_{1,b}(x+h)dh_{1}dh_{2}dh_{3}|\\
    &\leq\int_{|h_{1}^2+h_{2}^2|\geq M^{-1+\delta}} 1_{(|x_{1}^2+x_{2}^2|\leq 2)}M^{\epsilon-b}C_{k}h_{2}\sum_{i=0}^{k}\big(\frac{M^{\frac{i}{2}}}{|h_{1}^{2}+h_{2}^{2}|^{\frac{2+k-i}{2}}}\big)\frac{2^k}{M^k}dh_{1}dh_{2}\\
    &\leq C_{\delta,k} M^{-k\delta-1+\delta+\epsilon-b}\leq CM^{-\frac{3}{2}}.
\end{align*}

With that in mind, we can focus now on the integral only when $|h_{1}^2+h_{2}^2|\leq M^{-1+\delta}$. Next, using the properties of $g_{er,1,b}(x_{2})$ we check that
\begin{align*}
    &|\int_{|h_{1}^2+h_{2}^2|\leq M^{-1+\delta},h_{3}\in\mathds{R}}\frac{h_{2}}{|h|^3}(\w_{1,b}(x+h)-\w_{1}(x_{1},x_{2}+h_{2},x_{3}+h_{3})dh_{1}dh_{2}dh_{3}|\\
    &\leq C||f(z)||_{C^1}M^{\epsilon-b} M^{\frac{1}{2}+b-\epsilon}\int_{|h_{1}^2+h_{2}^2|\leq M^{-1+\delta},h_{3}\in\mathds{R}}\frac{h_{1}h_{2}}{|h|^3}dh_{1}dh_{2}dh_{3}\\
    &\leq CM^{\frac{1}{2}}\int_{|h_{1}^2+h_{2}^2|\leq M^{-1+\delta}}dh_{1}dh_{2}\leq CM^{-\frac{3}{2}+2\delta}.\\
\end{align*}

Similarly, using the properties of $a_{b}(x_{2}),g_{er,1,b}(x_{2}),g_{2,b}(x_{2}),g_{er,3,b}(x_{2})$ and $f(z)$ we get 

\begin{align*}
    &|\int_{|h_{1}^2+h_{2}^2|\leq M^{-1+\delta},h_{3}\in\mathds{R}}\frac{h_{2}}{|h|^3}\Big(\w_{1,b}(x_{1},x_{2}+h_{2},x_{3}+h_{3})\\
    &-M^{\epsilon-b}a_{b}(x_{2})f(M^{\frac{1}{2}+b-\epsilon}x_{1}g_{er,1,b}(x_{2}))\sin(M(g_{2,b}(x_{2})+h_{2}\partial_{x_{2}}g_{2,b}(x_{2}))f(M^{\frac{1}{2}+\epsilon}g_{2,b}(x_{2}))\\
   &\times \sin(Mg_{er,3,b}(x_{2})(x_{3}+h_{3}))f(M^{\frac{1}{2}}g_{er,3,b}(x_{2})(x_{3}+h_{3})\Big)dh_{1}dh_{2}dh_{3}|\\
   &\leq \int_{|h_{1}^2+h_{2}^2|\leq M^{-1+\delta},h_{3}\in\mathds{R}}\frac{CM^{\epsilon-b}h_{2}}{|h|^3}h_{2}\Big( 1+h_{2}M^{\frac{1}{2}+b}+M^{\frac{1}{2}+\epsilon-b}+M^{b}\Big)dh_{1}dh_{2}dh_{3}\\
   &\leq C M^{\epsilon-b}\int_{|h_{1}^2+h_{2}^2|\leq M^{-1+\delta}}\frac{h_{2}^2}{h_{1}^2+h_{2}^2}\Big(h_{2}M^{\frac{1}{2}+b}+M^{\frac{1}{2}+\epsilon-b}+M^{b}\Big)dh_{1}dh_{2}\leq C M^{2\epsilon-\frac{3}{2}+2\delta}.
\end{align*}
Finally, using that
\begin{align*}
    &|\int_{\mathds{R}}\frac{1}{|h|^3} \sin(Mg_{er,3,b}(x_{2})(x_{3}+h_{3}))[f(M^{\frac{1}{2}}g_{er,3,b}(x_{2})(x_{3}+h_{3}))-f(M^{\frac{1}{2}}g_{er,3,b}(x_{2})x_{3})]dh_{3}|\\
    &\leq C\int_{-\infty}^{\infty}\frac{M^{\frac12}|h_{3}|}{|h|^3}dh_{3}\leq C\frac{M^{\frac12}}{(h_{1}^2+h_{2}^2)^\frac{1}{2}}
\end{align*}
and combining this with all the other properties, we get that
\begin{align*}
    &|u_{3}((\w_{1,b},0,0))-\int_{|h_{1}^2+h_{2}^2|\leq M^{-1+\delta},h_{3}\in\mathds{R}}M^{\epsilon-b}\frac{h_{2}}{|h|^3}a_{b}(x_{2})f(M^{\frac{1}{2}+b-\epsilon}x_{1}g_{er,1,b}(x_{2}))f(M^{\frac{1}{2}+\epsilon}g_{2,b}(x_{2}))\\
   &\times \sin(M(g_{2,b}(x_{2})+h_{2}\partial_{x_{2}}g_{2,b}(x_{2}))\sin(Mg_{er,3,b}(x_{2})(x_{3}+h_{3}))f(M^{\frac{1}{2}}g_{er,3,b}(x_{2})x_{3})dh_{1}dh_{2}dh_{3}|\\
   &\leq C_{\delta}M^{2\epsilon-\frac{3}{2}+2\delta}\\
   &+\int_{|h_{1}^2+h_{2}^2|\leq M^{-1+\delta},h_{3}\in\mathds{R}}\frac{h_{2}}{|h|^3}M^{\epsilon-b}|a_{b}(x_{2})f(M^{\frac{1}{2}+b-\epsilon}x_{1}g_{er,1,b}(x_{2}))\sin(M(g_{2,b}(x_{2})+h_{2}\partial_{x_{2}}g_{2,b}(x_{2}))\\
   &\times f(M^{\frac{1}{2}+\epsilon}g_{2,b}(x_{2}))\sin(Mg_{er,3,b}(x_{2})(x_{3}+h_{3}))[f(M^{\frac{1}{2}}g_{er,3,b}(x_{2})(x_{3}+h_{3})-f(M^{\frac{1}{2}}g_{er,3,b}(x_{2})x_{3})]dh_{1}dh_{2}dh_{3}|\\
   &\leq C_{\delta}M^{2\epsilon-\frac{3}{2}+2\delta}+C\int_{|h_{1}^2+h_{2}^2|\leq M^{-1+\delta}}M^{\epsilon-b}\frac{h_{2}M^{\frac12}}{(h_{1}^2+h_{2}^2)^\frac{1}{2}}dh_{1}dh_{2}\leq C_{\delta}M^{2\epsilon-\frac{3}{2}+2\delta}.\\
\end{align*}

This means that, to obtain the expression for $\tilde{u}_3$, it is enough to study

\begin{align*}
    &\int_{|h_{1}^2+h_{2}^2|\leq M^{-1+\delta},h_{3}\in\mathds{R}}M^{\epsilon-b}a_{b}(x_{2})f(M^{\frac{1}{2}+b-\epsilon}x_{1}g_{er,1,b}(x_{2}))\sin(M(g_{2,b}(x_{2})+h_{2}\partial_{x_{2}}g_{2,b}(x_{2}))\\
   &\times f(M^{\frac{1}{2}+\epsilon}g_{2,b}(x_{2}))\sin(Mg_{er,3,b}(x_{2})(x_{3}+h_{3}))f(M^{\frac{1}{2}}g_{er,3,b}(x_{2})x_{3})dh_{1}dh_{2}dh_{3}
\end{align*}
and, in particular, since most integrands do not depend on $h_{1}$, $h_{2}$ or $h_{3}$ it is enough to study
\begin{align*}
    &\int_{|h_{1}^2+h_{2}^2|\leq M^{-1+\delta},h_{3}\in\mathds{R}}\frac{h_{2}}{|h|^3}\sin(M(g_{2,b}(x_{2})+h_{2}\partial_{x_{2}}g_{2,b}(x_{2})) \sin(Mg_{er,3,b}(x_{2})(x_{3}+h_{3}))dh_{1}dh_{2}dh_{3}\\
    &=\sin(Mg_{er,3,b}(x_{2})x_{3})\cos(M(g_{2,b}(x_{2}))\int_{|h_{1}^2+h_{2}^2|\leq M^{-1+\delta},h_{3}\in\mathds{R}}\frac{h_{2}}{|h|^3}\sin(Mh_{2}\partial_{x_{2}}(g_{2,b})(x_{2})) \\
    &\times\cos(Mg_{er,3,b}(x_{2})h_{3}))dh_{1}dh_{2}dh_{3}\\
    &=\frac{\sin(Mg_{er,3,b}(x_{2})x_{3})\cos(M(g_{2,b}(x_{2}))}{M}\int_{|h_{1}^2+h_{2}^2|\leq M^{\delta},h_{3}\in\mathds{R}}\frac{h_{2}}{|h|^3}\sin(h_{2}(\partial_{x_{2}}g_{2,b})(x_{2})) \\
    &\times\cos(g_{er,3,b}(x_{2})h_{3}))dh_{1}dh_{2}dh_{3},
\end{align*}
so we will study
$$H_{\lambda}:=\int_{|h_{1}^2+h_{2}^2|\leq \lambda,h_{3}\in\mathds{R}}\frac{h_{2}}{|h|^3}\sin(h_{2}\partial_{x_{2}}(g_{2,b})(x_{2}))\cos(g_{er,3,b}(x_{2})h_{3}))dh_{1}dh_{2}dh_{3}$$
Note that if we show that
$$\text{lim}_{\lambda\rightarrow{\infty}} H_{\lambda}=C_{0}\frac{\partial_{x_{2}}(g_{2,b})(x_{2})}{(g_{3,er,b}(x_{2}))^2+(\p_{x_{2}}g_{2,b}(x_{2}))^2}$$
and
\begin{equation}\label{HM1}
    |H_{\lambda_{1}}-H_{\lambda_{2}}|\leq C [\text{min}(\lambda_{1},\lambda_{2})]^{-\frac{1}{2\delta}},
\end{equation}
then we have showed the desired bound for \eqref{linftu3} with $j=0$. 

To show \eqref{HM1}, we note that by using integration by parts with respect to $h_{3}$ $k$ times to get

$$|H_{\lambda_{2}}-H_{\lambda_{1}}|\leq \int_{\lambda_{1}\leq |h_{1}^2+h_{2}^2|\leq \lambda_{2},h_{3}\in\mathds{R}}\frac{C_{k}}{|h|^{2+k}}dh_{1}dh_{2}dh_{3}\leq C_{k}[\text{min}(\lambda_{1},\lambda_{2})]^{-k+1}$$
so taking $k$ big gives us \eqref{HM1}. Note that in particular this shows that the limit of $H_{\lambda}$ when $\lambda$ tends to infinity exists.

Then we can use integration by parts with respect to $h_{3}$ to show that
$$H_{\lambda}=\frac{1}{g_{er,3,b}(x_{2})}\int_{|h_{1}^2+h_{2}^2|\leq \lambda,h_{3}\in\mathds{R}}\frac{3h_{2}h_{3}}{|h|^5}\sin(h_{2}\partial_{x_{2}}(g_{2,b})(x_{2}))\sin(g_{er,3,b}(x_{2})h_{3}))dh_{1}dh_{2}dh_{3}.$$
Now, if we define
$$\tilde{H}_{\lambda}=\frac{1}{g_{er,3,b}(x_{2})}\int_{|h_{1}^2|\leq \lambda,\ h_{2},h_{3}\in\mathds{R}}\frac{3h_{2}h_{3}}{|h|^5}\sin(h_{2}\partial_{x_{2}}(g_{2,b})(x_{2}))\sin(g_{er,3,b}(x_{2})h_{3}))dh_{1}dh_{2}dh_{3},$$
which is well defined since $\frac{h_{2}h_{3}}{|h|^5}$ has enough decay at infinity, using integration by parts with respect to $h_{3}$ once we have that

\begin{align*}
   &|H_{\lambda}-\tilde{H}_{\lambda}|\leq |\frac{1}{g_{er,3,b}(x_{2})}\int_{|h_{1}^2|\leq \lambda,|h_{1}^2+h_{2}^2|\geq \lambda\,h_{3}\in\mathds{R}}\frac{3h_{2}h_{3}}{|h|^5}\sin(h_{2}\partial_{x_{2}}(g_{2,b})(x_{2}))\sin(g_{er,3,b}(x_{2})h_{3}))dh_{1}dh_{2}dh_{3}|\\ 
   &\leq \frac{C}{|g_{er,3,b}(x_{2})|}\int_{|h_{1}^2+h_{2}^2|\geq \lambda}\frac{1}{|h_{1}^2+h_{2}^2|^\frac{3}{2}}dh_{1}dh_{2}\leq C \lambda
\end{align*}
so that in particular 

$$\text{lim}_{\lambda\rightarrow \infty}H_{\lambda}=\text{lim}_{\lambda\rightarrow \infty}\tilde{H}_{\lambda}.$$
But, if we define $R:=(g_{er,3,b}(x_{2})^2+\partial_{x_{2}}(g_{2,b})(x_{2})^2)^{\frac{1}{2}}$, and using the change of variables $Rz_{1}=h_{2}\partial_{x_{2}}(g_{2,b})(x_{2})+g_{er,3,b}(x_{2})h_{3}),Rz_{2}=-h_{3}\partial_{x_{2}}(g_{2,b})(x_{2})+g_{er,3,b}(x_{2})h_{2}$, and using that the integrands with the wrong parity (in $h_{2},h_{3},z_{1}$ or $z_{2}$) cancel out when we integrate, we have
\begin{align*}
    &\int_{|h_{1}^2|\leq \lambda,\ h_{2},h_{3}\in\mathds{R}}\frac{3h_{2}h_{3}}{|h|^5}\sin(h_{2}\partial_{x_{2}}(g_{2,b})(x_{2}))\sin(g_{er,3,b}(x_{2})h_{3}))dh_{1}dh_{2}dh_{3}\\
    &=-\text{PV}\int_{|h_{1}^2|\leq \lambda,\ h_{2},h_{3}\in\mathds{R}}\frac{3h_{2}h_{3}}{|h|^5}\cos(h_{2}\partial_{x_{2}}(g_{2,b})(x_{2})+g_{er,3,b}(x_{2})h_{3}))dh_{1}dh_{2}dh_{3}\\
    &=-\frac{g_{er,3,b}(x_{2})\partial_{x_{2}}(g_{2,b})(x_{2})}{R^2}\text{PV}\int_{|h_{1}^2|\leq \lambda,\ z_{1},z_{2}\in\mathds{R}}\frac{3(z_{1}^2-z_{2}^2)}{|h_{1}^2+z_{1}^2+z_{2}^2|^5}\cos(Rz_{1})dh_{1}dz_{1}dz_{2}\\
    &=-\frac{g_{er,3,b}(x_{2})\partial_{x_{2}}(g_{2,b})(x_{2})}{R^2}\text{PV}\int_{|\tilde{h}_{1}^2|\leq R^2\lambda,\tilde{h}_{2},\tilde{h}_{3}\in\mathds{R}}\frac{3(\tilde{h}_{2}^2-\tilde{h}_{3}^2)}{|\tilde{h}_{1}^2+\tilde{h}_{2}^2+\tilde{h}_{3}^2|^5}\cos(\tilde{h}_{2})d\tilde{h}_{1}d\tilde{h}_{2}d\tilde{h}_{3}
\end{align*}
and, after relabelling, if we denote
$$C_{0}:=-\text{PV}\int_{|h_{1}^2|\leq R^2\lambda,\ z_{1},z_{2}\in\mathds{R}}\frac{3(h_{2}^2-h_{3}^2)}{|h_{1}^2+h_{2}^2+h_{3}^2|^5}\cos(h_{2})dh_{1}dh_{2}dh_{3}$$
we have that
$$\text{lim}_{\lambda\rightarrow \infty}H_{\lambda}= C_{0}\frac{\partial_{x_{2}}(g_{2,b})(x_{2})}{(g_{3,er,b}(x_{2}))^2+(\p_{x_{2}}g_{2,b}(x_{2}))^2}$$
as we wanted. The only remaining thing to show \eqref{linftu3} is to prove that $C_{0}>0$. 
%
 For this, first we note that since $C_{0}$ does not depend on $g_{er,3,b}(x_{2})$ or $\partial_{x_{2}}(g_{2,b})(x_{2})$, we can take them to be equal to $1$. Next, we define

 \begin{align*}
     \bar{H}_{\lambda}=3\int_{h_{1}\in\mathds{R},|h_{2}^2|,|h_{3}^2|\leq \lambda}\frac{h_{2}h_{3}}{|h|^5}\sin(h_{2})\sin(h_{3}))dh_{1}dh_{2}dh_{3}
 \end{align*}
 and we can check that, for $\lambda_{n}=((n+\frac{1}{2})\pi )^2$, $\text{lim}_{n\rightarrow\infty}\tilde{H}_{\lambda_{n}}=\text{lim}_{n\rightarrow\infty}\bar{H}_{\lambda_{n}}$ since 
for $\lambda_{n}$ as defined earlier we have, we can use integration by parts with respect to $h_{3}$ to get
\begin{align*}
    &\text{lim}_{\lambda_{n}\rightarrow\infty}\int_{h^2_{1}\geq \lambda_{n} , h_{2}^2,h_{3}^2\leq \lambda_{n}}\frac{h_{2}h_{3}}{|h|^5}\sin(h_{2})\sin(h_{3})dh_{1}dh_{2}dh_{3}|\\
    &\leq \text{lim}_{\lambda_{n}\rightarrow\infty}\int_{h^2_{1}\geq \lambda_{n}, h_{2}^2,h_{3}^2\leq \lambda_{n}}\frac{1}{|h|^4}dh_{1}dh_{2}dh_{3}=0
\end{align*}

\begin{align*}
    &\text{lim}_{\lambda_{n}\rightarrow\infty}\int_{h_{2}^2,h_{3}^2\geq \lambda_{n} ,h^2_{1} \leq \lambda_{n}}\frac{h_{2}h_{3}}{|h|^5}\sin(h_{2})\sin(h_{3})dh_{1}dh_{2}dh_{3}|\\
    &\leq \text{lim}_{\lambda_{n}\rightarrow\infty}\int_{h_{2}^2,h_{3}^2\geq \lambda_{n} ,h^2_{1} \leq \lambda_{n}}\frac{1}{|h|^4}dh_{1}dh_{2}dh_{3}=0.
\end{align*}

With this in mind, we note that
\begin{align*}
    &\bar{H}_{\lambda_{n}}=3\int_{h_{1}\in\mathds{R},|h_{2}^2|,|h_{3}^2|\leq \lambda}\frac{h_{2}h_{3}}{|h|^5}\sin(h_{2})\sin(h_{3})dh_{1}dh_{2}dh_{3}\\
    &=3\int_{|h_{2}^2|,|h_{3}^2|\leq \lambda_{n}}\sin(h_{2})\sin(h_{3})\frac{h_{2}h_{3}}{h_{2}^2+h_{3}^2}\Big(\int_{-\infty}^{\infty}\frac{1}{|1+\frac{h_{1}^2}{h_{2}^2+h_{3}^2}|^5}\frac{1}{|h_{2}^2+h_{3}^2|^{\frac{1}{2}}}dh_{1}\Big)dh_{2}dh_{3}\\
    &=C\int_{|h_{2}^2|,|h_{3}^2|\leq \lambda_{n}}\sin(h_{2})\sin(h_{3})\frac{h_{2}h_{3}}{(h_{2}^2+h_{3}^2)^2}dh_{2}dh_{3}\\
    &=C\int_{0}^{\frac{\pi}{2}}\int_{0}^{\lambda_{n}L(\alpha)}\sin(r\sin(\alpha))\sin(r\cos(\alpha))\frac{r\sin(\alpha)r\cos(\alpha)}{r^4}rdrd\alpha\\
    &=C\int_{0}^{\frac{\pi}{2}}\sin(\alpha)\cos(\alpha)\int_{0}^{\lambda_{n}L(\alpha)}\frac{\cos(r(\sin(\alpha)-\cos(\alpha)))-\cos(r(\sin(\alpha)+\cos(\alpha)))}{r}drd\alpha\\
\end{align*}
where $C>0$ changes from line to line, we changed to polar coordinates in the fourth line and $L(\alpha)$ is the function that, given $\alpha\in[0,\frac{\pi}{2}]$ gives us the maximum value of $r$ such that $(r\sin(\alpha),r\cos(\alpha))\in [0,1]\times[0,1]$. But now, if we define
\begin{align*}
    &G(K,A,B):=\text{PV}\int_{0}^{K}\frac{\cos(rA)-\cos(rB)}{r}dr\\
    &=C_{A,B}-(\text{PV}\int_{K}^{\infty}\frac{\cos(rA)}{r}dr-\int_{K}^{\infty}\frac{\cos(rB)}{r}dr)\\
    &=C_{A,B}+\int_{KA}^{KB}\frac{\cos(r)}{r}dr
\end{align*}
which in particular gives, by taking the limit when $K$ is small, that $C_{A,B}=\ln(B)-\ln(A)$. Note also that $|G(K,A,B)|\leq 2(\ln(A)-\ln(B))$ and integration by parts gives us that, for $A,B>0$
$$\text{lim}_{K\rightarrow\infty}\int_{KA}^{KB}\frac{\cos(r)}{r}d=0$$
so we can apply the dominated convergence theorem to 
$$G_{\lambda_{n}}(\alpha):=\int_{0}^{\lambda_{n}L(\alpha)}\frac{\cos(r(\sin(\alpha)+\cos(\alpha)))-\cos(r(\sin(\alpha)-\cos(\alpha)))}{r}dr$$
to get
$$\frac{C_{0}}{2}=\text{lim}_{\lambda_{n}\rightarrow\infty}H_{\lambda_{n}}=C\int_{0}^{\frac{\pi}{2}}\sin(\alpha)\cos(\alpha)\ln(\frac{\sin(\alpha)+\cos(\alpha)}{|\sin(\alpha)-\cos(\alpha)|})d\alpha> 0$$
as we wanted.

For the bounds in $C^1$, the exact same  steps we followed to obtain the $L^{\infty}$ bound \eqref{linftu3}  allow us to get

$$||u_{j}(\p_{x_{i}}\w_{1,b},0,0)-\tilde{u}_{j}(\p_{x_{i}}\w_{1,b})||_{L^{\infty}}\leq C_{\delta}M^{-\frac{1}{2}+2\delta}$$
for $i=1,2,3$ and $j=2,3$, since we can decompose $\p_{x_{i}}\w_{1,b}$ in functions with the same structure as $\w_{1,b}$ (by aplying Leibniz's rule and looking at each of the individual terms obtained). Furthermore, since $\p_{x_{i}}u(\w)=u(\p_{x_{i}}\w)$, we have

$$||\p_{x_{i}}u_{j}(\w_{1,b},0,0)-\p_{x_{i}}\tilde{u}_{j}(\w_{1,b})||_{L^{\infty}}\leq ||u_{j}(\p_{x_{i}}\w_{1,b},0,0)-\tilde{u}_{j}(\p_{x_{i}}\w_{1,b})||_{L^{\infty}}+||\tilde{u}_{j}(\p_{x_{i}}\w_{1,b})-\p_{x_{i}}\tilde{u}_{j}(\w_{1,b})||_{L^{\infty}}$$
and thus to obtain the $C^1$ bounds it is enough to prove that
$$||\tilde{u}_{j}(\p_{x_{i}}\w_{1,b})-\p_{x_{i}}\tilde{u}_{j}(\w_{1,b})||_{L^{\infty}}\leq C_{\delta}M^{-\frac{1}{2}+2\delta}.$$

But (focusing on $\tilde{u}_{3}$, $\tilde{u}_{2}$ being analogous)

\begin{align*}
    &|\tilde{u}_{3}(\p_{x_{i}}\w_{1,b})-\p_{x_{i}}\tilde{u}_{3}(\w_{1,b})|\\
    &=|C_{0}(\p_{x_{i}}[\frac{M\p_{x_{2}}g_{2,b}(x_{2})}{(Mg_{3,er,b}(x_{2}))^2+(M\p_{x_{2}}g_{2,b}(x_{2}))^2}])M^{\epsilon-b}a_{b}(x_{2})f(M^{\frac{1}{2}+b-\epsilon}x_{1}g_{er,1,b}(x_{2}))\\
    &\times \cos(Mg_{2,b}(x_{2}))f(M^{\frac{1}{2}+\epsilon}g_{2,b}(x_{2})) \sin(Mg_{er,3,b}(x_{2})x_{3})f(M^{\frac{1}{2}}g_{er,3,b}(x_{2})x_{3})|\\
    &\leq C|(\p_{x_{i}}[\frac{\p_{x_{2}}g_{2,b}(x_{2})}{g_{3,er,b}(x_{2})^2+\p_{x_{2}}g_{2,b}(x_{2})^2}])|M^{\epsilon-b-1}\leq CM^{\epsilon-1}.\\
\end{align*}
For the bounds with the $\p_{x_{i}}$ derivative, $i=1,3$, the proof is exactly the same, since taking a derivative with respect to $x_{1}$ or $x_{3}$ gives us a function with the same properties. Relabeling $\delta$ in the error bounds so that $\delta_{new}=2\delta_{old}$ finishes the proof.





\end{proof}


\begin{lemma}\label{cuaderror}
    If we have $N,M>0$ big enough, then for $t\in[1-N^{-\frac{\epsilon}{2}},1]$, $\alpha\in[0,1]$
    $$||\w\cdot\nabla u(\w)||_{C^{\alpha}},||u(\w)\cdot\nabla \w||_{C^{\alpha}}\leq CM^{3\epsilon-\frac{1}{2}+\alpha}.$$
\end{lemma}

\begin{proof}

First we note that, using the properties of $\w_{1,b},\w_{3,b}$
we have that, for $i=0,1,2$
\begin{equation}\label{w1bw3b}
    ||\w_{1,b}||_{C^{i}}\leq \frac{CM^{\epsilon+i}}{K_{N}(t)},||\w_{3,b}||_{C^{i}}\leq CM^{i-\frac{1}{2}}.
\end{equation}
Furthermore, since we can apply lemma \ref{vaprox}, using the properties of $\tilde{u}(\w_{1,b})$ and of $\tilde{w}_{1,b}$  plus some direct computations gives us for $i=0,1$
$$||\tilde{u}_{2}(\w_{1,b})\p_{x_{2}}\w_{1,b}+\tilde{u}_{3}(\w_{1,b})\p_{x_{3}}(w_{1,b})||_{C^{i}}\leq CM^{3\epsilon-2b-\frac{1}{2}+i}$$
since the biggest terms, of order $M^{2\epsilon-2b+i}$, cancel out. On the other hand, just the bounds for $\tilde{u}$ and $\w_{3,b}$ gives us

$$||\tilde{u}_{2}(\w_{1,b})\p_{x_{2}}\w_{3,b}+\tilde{u}_{3}(\w_{1,b})\p_{x_{3}}(\w_{3,b})||_{C^{i}}\leq CM^{\epsilon-b-\frac{1}{2}+i}.$$

But, since, for $j=2,3$, $i=0,1$, for any $\delta>0$
$$||\tilde{u}_{j}(\w_{1,b})-u_{j}(\w_{1,b})||_{C^{i}}\leq C_{\delta}M^{2\epsilon-\frac{3}{2}+\delta}$$
we have, for $i=0,1$
$$||(\tilde{u}_{j}(\w_{1,b})-u_{j}(\w_{1,b}))\p_{x_{j}}\w_{1,b}||_{C^{i}}\leq C_{\delta} M^{3\epsilon-\frac{1}{2}+i+\delta-b},||(\tilde{u}_{j}(\w_{1,b})-u_{j}(\w_{1,b}))\p_{x_{j}}\tilde{w}_{3,b}||_{C^{i}}\leq C_{\delta} M^{2\epsilon-1+i+\delta},$$
and combining the bounds we get, for  $i=0,1$
\begin{equation}\label{uw1}
    ||u(\w_{1}(x,t))\cdot\nabla \w(x,t)||_{C^{i}}\leq C_{\delta}M^{3\epsilon-\frac{1}{2}+i+\delta}
\end{equation}

Regarding $\w_{3,b}$, since we have

$$\w_{3,b}=M^{\frac{1}{2}}p(x_{1},x_{2})\int_{-\infty}^{x_{3}}\sin(Mg_{er,3,b}(x_{2})z)f(M^{\frac{1}{2}}g_{er,3,b}(x_{2})z)dz$$
which, using integration by parts with respect to $x_{3}$ can be written as
$$M^{-\frac{1}{2}} p(x_{1},x_{2})[\sum_{i=0}^{k}\frac{p_{i}(x_{2},x_{3})}{M^{\frac{i}{2}}}\cos(Mg_{er,3,b}(x_{2})x_{3}+\frac{i\pi}{2})+p^{error}_{k}(x_{2},x_{3})]$$
with $||p(x_{1},x_{2})||_{L^{\infty}}\leq C,||p_{i}(x)||_{L^{\infty}}\leq C_{i},p_{i}(x)\in C^{\infty}, p^{error}_{i}(x_{2},x_{3})\leq C_{i}M^{-\frac{i}{2}}$ and $p,p^{error}_{k}$ and $p_{i}$ supported in $B_{1}(0).$

We can then act as in \eqref{h3small}, to obtain that, for $j=1,2$, for any $a>1$, $\delta>0$

\begin{align*}
    &\int_{|h_{1}|,|h_{2}|\geq M^{-1+\delta}}\int_{\mathds{R}}\frac{h_{j}}{|h|^3}p(x_{1}+h_{1},x_{2}+h_{2})\frac{p_{i}(M^{\frac{1}{2}}g_{er,3,b}(x_{2}+h_{2})(x_{3}+h_{3}))}{M^{\frac{i}{2}}}\\
    &\times\cos(Mg_{er,3,b}(x_{2}+h_{2})(x_{3}+h_{3})+\frac{i\pi}{2})dh_{1}dh_{2}dh_{3}\leq \frac{C_{a,\delta}}{M^{a}},
\end{align*}
and furthermore
\begin{align*}
    &\int_{|h_{1}|,|h_{2}|\geq M^{-1+\delta}}\int_{\mathds{R}}\frac{h_{j}}{|h|^3}p(x_{1}+h_{1},x_{2}+h_{2})p^{error}_{k}(x_{2}+h_{2},x_{3}+h_{3}) dh_{1}dh_{2}dh_{3}\leq \frac{C_{i}}{M^{\frac{i}{2}}},
\end{align*}
so, focusing for example on $u_{1}(0,0,\w_{3,b})$, $u_{2}$ being analogous, we have 

\begin{align*}
    &|u_{1}(0,0,\w_{3,b})|\leq M^{-\frac{1}{2}}(|\int_{|h_{1}|,|h_{2}|\leq M^{-1+\delta}}\int_{\mathds{R}}\frac{h_{2}}{|h|^3}\w_{3}(x+h,t)dh_{3}dh_{1}dh_{2}|+\frac{C}{M})\\
    &\leq CM^{-\frac{1}{2}}(|\int_{|h_{1}|,|h_{2}|\leq M^{-1+\delta}}\frac{1}{|h|}dh_{1}dh_{2}|+\frac{C}{M})\leq C_{\delta}M^{-\frac{3}{2}+\delta}.
\end{align*}
Furthermore, we can use that $\p_{x_{i}}u_{j}$ is a singular integral operator  to obtain, for $\alpha \in(0,1)\cup(1,2)$
$$||\p_{x_{i}}u_{j}(0,0,\w_{3,b})||_{C^\alpha}\leq C||\w_{3,b}||_{C^{\alpha}}\leq CM^{-\frac{1}{2}+\alpha}$$
and combining this with the $L^{\infty}$ bound gives us, for $i=0,1,2$, $\delta>0$
\begin{equation}\label{w3i}
   ||u_{j}(0,0,\w_{3,b})||_{C^{i}}\leq C_{\delta}M^{-\frac{3}{2}+i+\delta} .
\end{equation}
Then, the bounds for $\w_{j}(x,t)$ and its two first derivatives gives us, for $j=1,2,3$, $i=0,1$

\begin{equation}\label{cuaduw3}
    ||u_{j}(\w_{3,b})\cdot\nabla \w(x,t)||_{C^{i}}\leq CM^{-\frac{3}{2}+\delta+\epsilon+i}.
\end{equation}
Combining \eqref{cuaduw3} with \eqref{uw1}, taking $\delta=\epsilon$ and using interpolation gives us for $[0,1]$, 

$$||u(\w(x,t))\cdot\nabla \w(x,t)||_{C^{\alpha}}\leq C M^{4\epsilon-\frac{1}{2}+\alpha}.$$

On the other hand, using lemma \ref{vaprox} we have, for  $i=0,1$, $k=2,3$ that

$$||\p_{x_{1}}u_{k}(\w_{1,b})||_{C^{i}}\leq ||\p_{x_{1}}\tilde{u}_{k}(\w_{1,b})||_{C^{i}}+||\p_{x_{1}}(u_{k}-\tilde{u}_{k})(\w_{1,b})||_{C^{i}}\leq C M^{2\epsilon-\frac{1}{2}+\delta+i}$$
and using that $||\p_{x_{3}}u_{k}||$ is a singular integral operator, the bounds for $w_{1,b}$ and interpolation we get
$$||\p_{x_{3}}u_{k}(\w_{1,b})||_{C^{i}}\leq C M^{\epsilon-b+i}$$
which give us, using the bounds for $\w_{1,b}$ and $\w_{3,b}$, for $i=0,1$

$$||\w_{1,b}\p_{x_{1}}u_{k}(\w_{1,b})||_{C^{i}}\leq CM^{3\epsilon-\frac{1}{2}+\delta+i}$$
$$||\w_{3,b}\p_{x_{3}}u_{k}(\w_{1,b})||_{C^{i}}\leq CM^{\epsilon-\frac{1}{2}+i}$$
and taking $\delta=\epsilon$ and using interpolation gives us that

$$||\w\cdot\nabla u(\w)||_{C^{\alpha}}\leq CM^{4\epsilon-\frac{1}{2}+\alpha}$$
as we wanted.

\end{proof}

In order to control interactions that come from our small layer moving the vorticity of bigger scale layers, we need to obtain bounds for the decay of the velocity generated by the small scale vorticity.

\begin{lemma}\label{decay}
If $M$ is big enough, we have that, if $|x_{1}|\geq4 M^{-\frac{1}{2}-b+\epsilon}$, then for $i=0,1,2$, $j=1,2,3$ and any $k>1$

\begin{equation}\label{decay1}
    |D^{i}u_{j}(\w_{1,b})(x)|\leq C_{k}M^{-k}.
\end{equation}
Similarly, if $|x_{2}|\geq 4M^{-\frac{1}{2}-\epsilon+b}$, $i=0,1,2$ then
\begin{equation}\label{decay2}
    |D^{i}u_{j}(\w_{1,b})(x)|\leq C_{k}M^{-k}.
\end{equation}


\end{lemma}

\begin{proof}
We begin by showing \eqref{decay1} with $i=0$ and in the case $|x_{1}|\geq4 M^{-\frac{1}{2}-b+\epsilon}$, the case $|x_{2}|\geq 4M^{-\frac{1}{2}-\epsilon+b}$ being analogous. Similarly, we consider only $j=2$. We first note that $$\text{supp}(\w_{1,b})\subset[-2M^{-\frac{1}{2}-b+\epsilon},2M^{-\frac{1}{2}-b+\epsilon}]\times[-2M^{-\frac{1}{2}-\epsilon+b},2M^{-\frac{1}{2}-\epsilon+b}]\times[2M^{-\frac12},2M^{-\frac12}].$$
But, acting exactly as in \eqref{h3small}, using integration by parts with respect to $h_{3}$ we can obtain, for $|h_{1}|\geq M^{-\frac{1}{2}-b+\epsilon}$, for any $m>1$

$$\int_{\mathds{R}}\frac{h_{2}}{|h|^3}\w_{1,b}(x+h)dh_{3}\leq \int_{\mathds{R}} \sum_{l=0}^{m}\frac{C_{m}M^{\frac{l}{2}}}{|h|^{2+m-l}} \frac{1}{M^{m}}dh_{3}\leq  \frac{C_{m}}{M^{\epsilon m}}.$$

Then, using the expression for $u_{2}(\w_{1,b})$ and the bounds for the support of $\w_{1,b}(x)$ we get, if $|x_{1}|\geq 4M^{-\frac12-b+\epsilon}$

$$|u(\w_{1,b})|\leq |\int_{\mathds{R}^3}\frac{h_{2}}{|h|^3}\w_{1,b}(x+h)dh_{3}dh_{2}dh_{1}|\leq |\int_{-2M^{-\frac{1}{2}+\epsilon}-x_{1}}^{-2M^{-\frac{1}{2}+\epsilon}-x_{1}}\int_{-2M^{-\frac{1}{2}-\epsilon+b}-x_{2}}^{2M^{-\frac{1}{2}-\epsilon+b}-x_{2}}\frac{C_{m}}{M^{\frac{m}{2}}}dh_{2}dh_{1}|\leq \frac{C_{m}}{M^{\epsilon m}},$$
and since $m$ is arbitrary this finishes the proof of \eqref{decay1}. For \eqref{decay2} the proof is completely analogous. To obtain the same result for $i=1,2$, we just use the fact that derivatives commute with the velocity operator and argue in the exact same way as before but with $u(D^{i}\w_{1,b})$


\end{proof}
This lemma now allows us to obtain bounds for the interactions that come from the velocity of the small scale layer interacting with the big scale layer.
\begin{lemma}\label{upequeñoagrande}
    Let $\w^{N}(x,t)=(\w^{N}_{1}(x,t)\w^{N}_{2}(x,t),\w^{N}_{3}(x,t))$ be an odd-vorticty such that $||\w^{N}(x,t)||_{C^{2.5}}\leq N^{4}$ for $t\in[1-N^{-\frac{\epsilon}{2}},1]$ and with $\text{supp}(\w^{N}(x,t))\subset |x|\leq 1$.
    Then, if $N$ is big enough, we have that, for $\alpha\in[0,1]$
    \begin{equation}\label{smallbig1}
        ||u(\w)\cdot\nabla\w^{N}||_{C^{\alpha}}\leq CM^{2\epsilon+\alpha-\frac{1}{2}}
    \end{equation}
    \begin{equation}\label{smallbig2}
        ||\w^{N}\cdot\nabla u(\w)||_{C^{\alpha}}\leq CM^{2\epsilon+\alpha-\frac{1}{2}}
    \end{equation}
    
\end{lemma}

\begin{proof}
    We start by noting that, using \eqref{w3i} and the bounds for $\w^{N}$, we directly have, for $N$ big enough so that $N^{4}\leq M^{\epsilon}$, and after taking $\delta=\epsilon$, for $i=0,1$
    \begin{equation}\label{3buw}
        ||u(0,0,\w_{3,b})\cdot\nabla\w^{N}||_{C^{i}}\leq C_{\delta} N^{4}M^{-\frac{3}{2}+i+\delta}\leq CM^{-\frac{3}{2}+i+2\epsilon},
    \end{equation}
    \begin{equation}\label{3bwu}
        ||\w^{N}\cdot\nabla u(0,0,\w_{3,b})||_{C^{i}}\leq C_{\delta} N^{4}M^{-\frac{1}{2}+i+\delta}\leq CM^{-\frac{1}{2}+i+2\epsilon}.
    \end{equation}

    Furthermore, since, for $i=0,1,$
    \begin{equation}\label{uw1b}
        ||u_{j}(\w_{1,b},0,0)||_{C^{i}}\leq ||(\tilde{u}_{j}-u_{j})(\w_{1,b},0,0)||_{C^{i}}+||\tilde{u}_{j}(\w_{1,b},0,0)||_{C^{i}}\leq CM^{\epsilon-b-1+i}
    \end{equation}
    we can get
    $$||u(\w_{1,b},0,0)\cdot\nabla\w^{N}||_{C^{i}}\leq C_{\delta} N^{4}M^{\epsilon-1+i}\leq C_{\delta}M^{2\epsilon-1+i},$$
    and combining this with \eqref{3buw} plus interpolation already yields \eqref{smallbig1}.
    
    On the other hand, using \eqref{uw1b}, we have
    \begin{equation}\label{D1wu}
        |D^{1}[\w^{N}(x)\cdot\nabla u(\w_{1,b},0,0)(x)]|\leq CM^{2\epsilon}+|\w^{N}(x)\cdot\nabla D^{1}[u(\w_{1,b},0,0)(x)]
    \end{equation}
    and, since for $x$ such that $|x_{1}|\geq 4M^{-\frac{1}{2}+\epsilon}$ or $|x_{2}|\geq 4M^{-\frac{1}{2}-\epsilon+b}$, we can apply \eqref{decay1} or \eqref{decay2} from lemma \ref{decay} to obtain, for $i=0,1$ that
    $$|\w^{N}(x)\cdot\nabla u(\w_{1,b},0,0)(x)|)\leq C_{k,\delta}M^{-k},|\w^{N}(x)\cdot\nabla D^{1}[u(\w_{1,b},0,0)(x)]|)\leq C_{k,\delta}M^{-k},$$
    we can combine this with \eqref{D1wu}, for $|x_{1}|\geq 4M^{-\frac{1}{2}+\epsilon}$ or $|x_{2}|\geq 4M^{-\frac{1}{2}-\epsilon+b}$ we get
    $$|\w^{N}(x)\cdot\nabla u(\w_{1,b},0,0)(x)|)\leq C_{k,\delta}M^{-k},|D^{1}[\w^{N}(x)\cdot\nabla u(\w_{1,b},0,0)(x)]|)\leq C_{k,\delta}M^{-k}+CM^{2\epsilon}.$$
But, for $|x_{1}|\leq 4M^{-\frac{1}{2}+\epsilon}\cap|x_{2}|\leq 4M^{-\frac{1}{2}-\epsilon+b}\cap |x_{3}|\leq 1$ we can use the fact that $\w^{N}$ is odd and $\w^{N}\in C^{2.5}$ to get that
$$|\w^{N}_{1}(x)|\leq N^{4}x_{2}x_{3}\leq N^{4}M^{-\frac{1}{2}+\epsilon},|\w^{N}_{2}(x)|\leq N^{4}x_{1}x_{3}\leq N^{4}M^{-\frac{1}{2}-\epsilon+b},|\w^{N}_{3}(x)|\leq N^{4}x_{1}x_{2}\leq N^{4}M^{-\frac{1}{2}},$$
so, for $|x_{1}|\leq 4M^{-\frac{1}{2}+\epsilon}\cap|x_{2}|\leq 4M^{-\frac{1}{2}-\epsilon+b}$ 
we have
$$|\w^{N}(x)\cdot\nabla u(\w_{1,b},0,0)(x)|)\leq CN^{4}[M^{-\frac{1}{2}+\epsilon}+M^{-\frac{1}{2}-\epsilon+b}]M^{\epsilon-b}\leq CM^{2\epsilon-\frac{1}{2}},$$
$$|\w^{N}(x)\cdot\nabla D^{1}[u(\w_{1,b},0,0)(x)]|)\leq CN^{4}[M^{-\frac{1}{2}+\epsilon}+M^{-\frac{1}{2}-\epsilon+b}]M^{\epsilon-b+1}\leq CM^{2\epsilon+\frac{1}{2}}$$
which, combined with \eqref{D1wu} gives us, for for $|x_{1}|\leq 4M^{-\frac{1}{2}+\epsilon}\cap|x_{2}|\leq 4M^{-\frac{1}{2}-\epsilon+b}$, for $i=0,1$
$$|D^{i}[\w^{N}(x)\cdot\nabla u(\w_{1,b},0,0)(x)]|\leq CM^{2\epsilon-\frac{1}{2}+i}$$
and since we had already proved it for for $|x_{1}|\geq 4M^{-\frac{1}{2}+\epsilon}$ or $|x_{2}|\geq 4M^{-\frac{1}{2}-\epsilon+b}$, the bound holds for any $x$, and thus, for $i=0,1$
$$||\w^{N}(x)\cdot\nabla u(\w_{1,b},0,0)(x)||_{C^{i}}\leq CM^{2\epsilon-\frac{1}{2}+i}$$
which combined with \eqref{3bwu} and using interpolation gives \eqref{smallbig2}.
\end{proof}

Finally, since we are using a simplified version of the velocity $\bar{u}$ instead of $u$, a force is also required to compensate for this difference, and we need to show that this force is almost $C^{\frac{1}{2}}$ in vorticity formulation.

\begin{lemma}\label{ubarraerror}
    For $N$ big enough and $\alpha\in[0,1]$, we have that
    \begin{equation}
        ||(u^{N}-\bar{u}^{N})\cdot\nabla\w||_{C^{\alpha}}\leq CM^{4\epsilon-\frac{1}{2}+\alpha},||\w\cdot\nabla(u^{N}-\bar{u}^{N})||_{C^{\alpha}}\leq CM^{3\epsilon-\frac{1}{2}+\alpha}
    \end{equation}
\end{lemma}

\begin{proof}
    We start by noting that, using the fact that $u^{N}(x,t)$ is odd and $||u^{N}(x,t)||_{C^{3.5}}\leq N^{4}$, we have that, for $i=1,2,3$, defining $z=|x_{1}^2+x_{3}^2|^{\frac{1}{2}}$
    $$|u^{N}_{i}(x,t)-\bar{u}^{N}_{i}(x,t)|\leq C N^4x_{i}|x_{1}^2+x_{3}^2|=CN^{4}x_{i}z^2.$$
    But, using that, for $N$ big,
    $$\text{supp}(\w_{b})\subset[-2M^{-\frac{1}{2}-b+\epsilon},2M^{-\frac{1}{2}-b+\epsilon}]\times[-2M^{-\frac{1}{2}-\epsilon+b},2M^{-\frac{1}{2}-\epsilon+b}]\times[2M^{-\frac12},2M^{-\frac12}],$$
    we have that
    $$|(u^{N}_{i}(x,t)-\bar{u}^{N}_{i}(x,t))\p_{x_{i}}\w_{j,b}(x)|\leq C N^{4}|z^2x_{i}\p_{x_{i}}\w_{j,b}(x)|\leq CN^{4}M^{-1+2\epsilon}|x_{i}\p_{x_{i}}\w_{j,b}|.$$

    But, using the properties of $\w_{j,b}$, it is easy to check that
    $|x_{i}\p_{x_{i}}\w_{j,b}(x)|\leq C M^{\epsilon-b} M^{\frac{1}{2}}$ so, taking $N$ big enough so that $N^{4}\leq M^{\epsilon}$
    $$|(u^{N}_{i}(x,t)-\bar{u}^{N}_{i}(x,t))\p_{x_{i}}\w_{j,b}(x)|\leq C N^{4}M^{-\frac{1}{2}+3\epsilon-b}\leq CM^{-\frac{1}{2}+4\epsilon}.$$
    A similar computation gives us
    $$|(u^{N}_{i}(x,t)-\bar{u}^{N}_{i}(x,t))D^{1}[\p_{x_{i}}\w_{j,b}(x)]|\leq CM^{\frac12+4\epsilon},$$
    which combined with
    $$|D^{1}[u^{N}_{i}(x,t)-\bar{u}^{N}_{i}(x,t)]\p_{x_{i}}\w_{j,b}(x)|\leq CN^{4}[z|x_{i}\p_{x_{i}}\w_{j,b}(x)|+z^2|\p_{x_{i}}\w_{j,b}(x)|]\leq CM^{4\epsilon}$$
    and using interpolation gives us
    $$||(u^{N}-\bar{u}^{N})\cdot\nabla\w||_{C^{\alpha}}\leq CM^{4\epsilon-\frac{1}{2}+\alpha}.$$
    On the other hand, we have
    $$|\w_{i,b}\p_{x_{i}}(u^{N}_{j}-\bar{u}^{N}_{j})|\leq CN^4|\w_{i,b}(x)|[z|x_{i}|+z^2]|\leq CM^{3\epsilon-\frac{1}{2}},$$
    $$|D^{1}[\w_{i,b}]\p_{x_{i}}(u^{N}_{j}-\bar{u}^{N}_{j})|\leq CN^4|D^1[\w_{i,b}(x)][z|x_{i}|+z^2]|\leq CM^{3\epsilon+\frac{1}{2}},$$
    $$|\w_{i,b}D^{1}[\p_{x_{i}}(u^{N}_{j}-\bar{u}^{N}_{j})]|\leq CN^4[\w_{i,b}(x)]\leq CM^{2\epsilon}$$
    and again interpolation gives us
    $$\|\w_{i,b}\p_{x_{i}}(u^{N}_{j}-\bar{u}^{N}_{j})||_{C^{\alpha}}\leq CM^{3\epsilon-\frac12+\alpha}.$$
    
\end{proof}

Finally, in order to be able to use induction when gluing different layers together, we want to show that the small scale layer has the desired properties so that it can act as the big scale layer in the next iteration. This includes regularity of $\w$ as well as properties of the velocity field.

\begin{lemma}\label{ultimolemma}
    For $N$ big enough, $t\in[1-M^{-\frac{\epsilon}{2}},1]$, we have that
    \begin{equation}\label{w25u35}
        ||\w(x,t)||_{C^{2.5}},||u(\w)(x,t)||_{C^{3.5}}\leq \frac{M^{4}}{2}.
    \end{equation}
    Furthermore, we have
    \begin{equation}\label{2c0}
        C_{0}M^{\epsilon}\geq \p_{x_{2}}u_{2}(\w(\cdot,t))(x=0)\geq \frac{C_{0}}{4}M^{\epsilon}
    \end{equation}
    and
    \begin{equation}\label{u1log}
        |\p_{x_{1}}u_{1}(\w(\cdot,t))(x=0)|\leq 1.
    \end{equation}
\end{lemma}

\begin{proof}

First, we note that, since for $t\in[1-M^{-\frac{\epsilon}{2}},1]$, we have $\text{supp}(\w(x,t))\subset B_{1}(0)$, the evolution of $\w(x,t)$ is equivalent to 
$$\p_{t}\omega+v\cdot\nabla \w=\w\cdot\nabla v.$$
with $v(x,t)=\bar{u}^{N}(x,t)G(|x|)$, with $G(|x|)=1$ if $|x|\leq 1$, $G(|x|)=0$ if $|x|\geq 2$ and $G\in C^{\infty}$. Note that in particular $||v(x,t)||_{C^{3.5}}\leq CN^{4}$.

Now, if we define
    $$\tilde{\w}_{i}(x,t)=\w(\phi(x,1,t),t)$$
    with
    $$\p_{t}\phi(x,1,t)=v(\phi(x,1,t),t)$$
    $$\phi(x,1,1)=x,$$
    we get
    \begin{equation*}
        \p_{t}\tilde{\w}_{i}(x,t)=\tilde{\w}_{i}(x,t)(\p_{x_{i}}v_{i})(\phi(x,1,t),t)
    \end{equation*}
    $$\tilde{\w}_{i}(x,t=1)=\w_{i,end}(x).$$
    
    But 
    \begin{align*}
        &\p_{t}||\tilde{\w}_{i}(x,t)||_{C^{2.5}}\leq ||\tilde{\w}_{i}(x,t)(\p_{x_{i}}v_{i})(\phi(x,1,t),t)||_{C^{2.5}}\leq C ||\tilde{\w}_{i}(x,t)||_{C^{2.5}}||(\p_{x_{i}}v_{i})(\phi(x,1,t),t)||_{C^{2.5}}\\
        &\leq C ||\tilde{\w}_{i}(x,t)||_{C^{2.5}} ||v(x,t)||_{C^{3.5}}[1+||\p_{x_{1}}\phi(x,1,t)||_{C^{1.5}}+||\p_{x_{2}}\phi(x,1,t)||_{C^{1.5}}]^{3}\\
        &\leq  C N^{4}||\tilde{\w}_{i}(x,t)||_{C^{2.5}} [1+||\p_{x_{1}}\phi(x,1,t)||_{C^{1.5}}+||\p_{x_{2}}\phi(x,1,t)||_{C^{1.5}}]^{3}
    \end{align*}
    On the other hand, we have
    $$\p_{t}D^1\phi(x,1,t)=D^1(v)(\phi(x,1,t),t)D^{1}(\phi)(x,1,t),$$
    so
    $$\p_{t}||D^{1}\phi(x,1,t)||_{L^{\infty}}\leq ||D^{1}v||_{L^{\infty}}||D^{1}\phi||_{L^{\infty}}\leq CN^{4}||\phi(x,1,t)||_{L^{\infty}}$$
    and similarly
    \begin{align*}
        &\p_{t}||D^3\phi(x,1,t)||_{L^{\infty}}=(||D_{1}\phi(x,1,t)||_{L^{\infty}}+||D^{3}\phi(x,1,t)||_{L^{\infty}})(1+||v||_{C^{3}})^3\\
        &\leq CN^{12}(||D_{1}\phi(x,1,t)||_{L^{\infty}}+||D^{3}\phi(x,1,t)||_{L^{\infty}})
    \end{align*}
    and integrating in time and taking $N$ big enough we get, for $t\in[1-M^{-\frac{\epsilon}{2}},1]$
    $$||D^{1}\phi(x,1,t)||_{C^2}\leq 1+CN^{12}M^{-\frac{\epsilon}{2}}\leq 2,$$
    Now integrating in time $\p_{t}||\tilde{w}(x,t)||_{C^{2.5}}$, for $N$ big and $t\in[1-M^{\frac{\epsilon}{2}}]$ we get
    $$||\tilde{\w}(x,t)||_{C^{2.5}}\leq 2||\tilde{\w}(x,1)||_{C^{2.5}}\leq CM^{2.5+\epsilon}\leq  M^{3}.$$
    But then, since

    $$||\w(x,t)||_{C^{2.5}}=||\tilde{\w}(\phi(x,t,1),t)||_{C^{2.5}}\leq C ||\tilde{\w}(x,t)||_{C^{2.5}}(1+||D^1\phi(x,t,1)||_{C^{2}})^3$$
    and using that we can obtain the same kind of bounds for $||D^{1}\phi(x,t,1)||_{C^2}$ as the ones we got for $||D^{1}\phi(x,1,t)||_{C^{2}}$, we have that
    $$||\w(x,t)||_{C^{2.5}}\leq CM^{3}\leq \frac{M^{4}}{2}.$$
    Furthermore, since $\p_{x_{i}}u_{j}$ is a singular integral operator, this implies
     $$||D^1 u(\w(x,t))||_{C^{2.5}}\leq C||\w(x,t)||_{C^{2.5}}\leq CM^{3}\leq \frac{M^{4}}{4},$$
     and since $\text{supp}(\w(x,t))\subset B_{1}(0)$ we also have
    $$||u(\w(x,t))||_{L^{\infty}}\leq C||\w(x,t)||_{C^{2.5}}\leq CM^{3}\leq \frac{M^{4}}{4},$$
    so we have \eqref{w25u35}.
    To prove \eqref{2c0}, we define $\bar{\w}(x,t)$ as

    $$\p_{t}\bar{\w}+u_{lin}\cdot\nabla \bar{\w}=\bar{\w}\cdot\nabla u_{lin},$$
    with
    $$u_{lin}(x,t)=(x_{1}\p_{x_{1}}u^{N}_{1}(x=0,t),x_{2}\p_{x_{2}}u^{N}_{2}(x=0,t)x_{3}\p_{x_{3}}u^{N}_{3}(x=0,t)),$$
    which fulfils, if we define $A_{i}(t)=\int_{t}^{1}\p_{x_{i}}u^{N}_{i}(x=0,s)ds$
    $$w_{1}(x,t)=\frac{M^{\epsilon}}{A_{1}(t)}f(M^{\frac{1}{2}-\epsilon}A_{1}(t)x_{1})\sin(MA_{2}(t)x_{2})f(M^{\frac{1}{2}+\epsilon}A_{2}(t)x_{2}) \sin(MA_{3}(t)x_{3})f(M^{\frac{1}{2}}A_{3}(t)x_{3}).$$

We can then apply lemma \ref{vaprox} to get that

$$||(u_{2}-\tilde{u}_{2})(\bar{\w}_{1}(x,t))||_{C^1}\leq C_{\delta}M^{2\epsilon-\frac{1}{2}+\delta}$$
and using the expression of $\tilde{u}(\bar{\w}_{1}(x,t))$ plus the bounds for $u^{N}(x,t)$ we get, for any $\delta>0$, if $N$ big enough
$$\frac{1-\delta}{2}C_{0}M^{\epsilon}\leq \p_{x_{2}}\tilde{u}_{2}(\bar{\w}_{1}(x,t))\leq \frac{1+\delta}{2}C_{0}M^{\epsilon}.$$
Then, taking $N$ big enough we get, again for any $\delta>0$

    $$\frac{1-2\delta}{2}C_{0}M^{\epsilon}\leq \p_{x_{2}}u_{2}(\bar{\w}_{1}(x,t))\leq C_{0}M^{\epsilon}\frac{1+2\delta}{2}.$$

Furthermore, using that $\p_{x_{i}}u_{j}$ as an operator is a singular integral operator, we can get

$$||\p_{x_{i}}u_{j}(0,0,\bar{\w}_{3}(x,t))||_{L^{\infty}}\leq CM^{\epsilon-\frac{1}{2}}$$
so that
\begin{equation}\label{3delta}
\frac{1-3\delta}{2}C_{0}M^{\epsilon}\leq \p_{x_{2}}u_{2}(\bar{\w}(x,t))\leq C_{0}M^{\epsilon}\frac{1+3\delta}{2}.
\end{equation}

To finish the proof of \eqref{2c0}, we just need to obtain bounds for

$$||\p_{x_{2}}u_{2}(\w(x,t)-\bar{\w}(x,t))||_{L^{\infty}}.$$

For this, if we define $W=\w(x,t)-\bar{\w}(x,t)$, we have that

$$\p_{t}W+(v-u_{lin})\cdot\nabla \w+u_{lin}\cdot \nabla W=W\cdot\nabla v+\bar{\w}\cdot\nabla (v-u_{lin})$$
and using the bounds for the support of $\w,\bar{\w}$, the definitions of $v$ and $u_{lin}$, the bounds for $u^{N}$ and that $u^{N}$ is an odd velocity, we get
$$||(v-u_{lin})\cdot\nabla \w||_{L^{\infty}},||\bar{\w}\cdot\nabla (v-u_{lin})||_{L^{\infty}}\leq CN^{4} M^{-\frac{1}{2}+4\epsilon}\leq CM^{-\frac{1}{2}+5\epsilon}.$$
Then, the evolution equation for $W$ gives us
$$||W||_{L^{\infty}}\leq CM^{-\frac{1}{2}+5\epsilon}.$$

Furthermore, we have that $||\w||_{C^{1}},||\bar{\w}||_{C^{1}}\leq CM^{1+\epsilon}$ so $||W||_{C^{1}}\leq CM^{1+\epsilon}$
and finally, using the properties of the operator $\p_{x_{2}}u_{2}$ we get
$$||\p_{x_{2}}u_{2}(W)||_{L^{\infty}}\leq C||W||_{L^{\infty}}(\ln(1+||W||_{C^{1}}))\leq CM^{-\frac{1}{2}+6\epsilon},$$
which combined with \eqref{3delta} gives us the bound for $\p_{x_{2}}u_{2}(\w)$.

To obtain \eqref{u1log}, we use the properties of $\p_{x_{1}}u_{1}$ to get

$$|\p_{x_{1}}u_{1}(\w)|=|\p_{x_{1}}u_{1}(0,0,\w_{3}(x,t))|\leq C||\w_{3}||_{L^{\infty}}(1+\ln(||\w||_{C^{1}}))\leq 1.$$

\end{proof}
\section{Constructing the solution}
Now, combining all the lemmas we have proved so far, we can get the theorem that will allow us to construct our solution iteratively.

\begin{theorem}\label{glue}
    Given $\frac{1}{100}\geq \epsilon>0$, for $N>1$ big enough, given an odd vorticity $\bar{\w}^{1}(x,t)$ satisfying the forced 3D-Euler equations  with force $\bar{F}^{1}(x,t)$ such that $||\bar{\w}^1||_{C^{2.5}},||u(\bar{\w}^1)||_{C^{3.5}}\leq N^4$, $\bar{F}^1\in C^{\frac{1}{2}-6\epsilon},$ and fulfilling, for $t\in[1-N^{-\frac{\epsilon}{2}},1]$
    $$C_{0}N^{\epsilon}\geq \p_{x_{1}}u_{1}(\bar{\w}^1)\geq \frac{C_{0}}{4}N^{\epsilon}, ||\p_{x_{3}}u_{3}(\bar{\w}^1)||_{L^{\infty}}\leq \ln(N)^{3},$$
    with $C_{0}$ the constant from lemma \ref{vaprox}, then, if we define
    $$M^{\frac{1}{2}}:=e^{\int_{1-N^{-\frac{\epsilon}{2}}}^{1}\p_{x_{1}}u_{1}(\bar{\w}^1)(s,0)ds}$$
    there is an odd vorticity $\bar{\w}^2(x,t)$ such that $\bar{\w}^1(x,t)+\bar{\w}^2(x,t)$ is a solution to the forced 3D-Euler with force $\bar{F}^1(x,t)+\bar{F}^2(x,t)$ fulfilling 
    \begin{equation}\label{cotas1}
        ||\bar{\w}^1(x,t)+\bar{\w}^2(x,t)||_{C^{2.5}},||u(\bar{\w}^1(x,t)+\bar{\w}^2(x,t))||_{C^{3.5}}\leq M^{4},||\bar{F}^2(x,t)||_{C^{\frac{1}{2}-6\epsilon}}\leq M^{-2\epsilon}
    \end{equation}
    and, for $t\in[1-M^{-\frac{\epsilon}{2}},1]$
    \begin{equation}\label{cotas2}
        C_{0}M^{\epsilon}\geq \p_{x_{2}}u_{2}(\bar{\w}^1+\bar{\w}^2)\geq \frac{C_{0}}{4}M^{\epsilon}, ||\p_{x_{1}}u_{1}(\bar{\w}^1+\bar{\w}^{2})||_{L^{\infty}}\leq \ln(M)^{3}.
    \end{equation}
    Furthermore, for $t\in[0,1-N^{-\frac{\epsilon}{2}}]$, $\bar{F}^2(x,t)=\bar{\w}^{2}(x,t)=0$.
\end{theorem}

\begin{proof}
    To obtain $\bar{\w}^2(x,t)$, we consider

    \begin{equation*}
    \p_{t}\omega+\bar{u}(\w^{1})\cdot\nabla \w=\w\cdot\nabla\bar{u}(\w^{1}),
\end{equation*}
with $\bar{u}(\w^{1})$ as in definition \ref{ubarra}, and
$$\w_{1}(x,t=1)=-M^{\epsilon}f(M^{\frac{1}{2}-\epsilon}x_{1})\sin(Mx_{2})f(M^{\frac{1}{2}+\epsilon}x_{2})\sin(Mx_{3})f(M^{\frac{1}{2}}x_{3})$$
$$\w_{2}(x,t=1)=0,$$
$$\w_{3}(x,t=1)=-\int_{-\infty}^{x_{3}}\p_{x_{1}}\w_{1}(x_{1},x_{2},s,t=1)ds,$$
with some fixed $f(z)\in C^{\infty}$, $1\geq f(z)\geq 0$, $f(z)=1$ if $|z|\leq \frac{1}{2}$, $f(z)=0$ if $|z|\geq 1$.
We then take 
$$\bar{\w}^{2}(x,t):= g(t)\w(x,t)$$
with

\begin{equation}
g(t)=
    \begin{cases}
        0&\text{if } t\leq 1-N^{-\frac{\epsilon}{2}}\\
        (t-1+N^{-\frac{\epsilon}{2}})M^{\epsilon}&\text{if }t\in[1-N^{-\frac{\epsilon}{2}},1-N^{-\frac{\epsilon}{2}}+M^{-\epsilon}]\\
        1&\text{if } t\geq 1-N^{-\frac{\epsilon}{2}}+M^{-\epsilon},
    \end{cases}
\end{equation}
so $\bar{\w}^{2}(x,t)=0$ if $t\leq 1-N^{-\frac{\epsilon}{2}}$.
For $\bar{w}^{2}(x,t)$ defined this way, we have that $\bar{\w}^{1}(x,t)+\bar{\w}^{2}(x,t)$ is a solution to the forced 3D-Euler equation with force
$$F(x,t)=\bar{F}^{1}(x,t)+\bar{F}^2(x,t)$$
with
\begin{align*}
    \bar{F}^{2}(x,t)=&g'(t)\w(x,t)+g(t)\big((u-\bar{u})(\bar{\w}^{1}(x,t))\cdot\nabla \w(x,t)-\w(x,t)\cdot\nabla[(u-\bar{u})(\bar{w}^{1})(x,t)]\big)\\
    &+g(t)\big(u(\w)\cdot\nabla \bar{\w}^1(x,t)-\bar{\w}^1(x,t)\cdot\nabla u(\w)(x,t)\big)\\
    &+g(t)^2\big(u(\w)\cdot\nabla \w(x,t)-\w(x,t)\cdot\nabla u(\w)(x,t)\big).
\end{align*}

Then, using lemmas \ref{ubarraerror}, \ref{upequeñoagrande} and \ref{cuaderror} gives us
$$||g(t)\big((u-\bar{u})(\bar{\w}^{1}(x,t))\cdot\nabla \w(x,t)-\w(x,t)\cdot\nabla[(u-\bar{u})(\bar{w}^{1})(x,t)]\big)||_{C^{\frac{1}{2}-6\epsilon}}\leq CM^{-2\epsilon},$$
$$||g(t)\big(u(\w)\cdot\nabla \bar{\w}^1(x,t)-\bar{\w}^1(x,t)\cdot\nabla u(\w)(x,t)\big)||_{C^{\frac{1}{2}-6\epsilon}}\leq CM^{-4\epsilon}$$
$$||g(t)^2\big(u(\w)\cdot\nabla \w(x,t)-\w(x,t)\cdot\nabla u(\w)(x,t)\big)||_{C^{\frac{1}{2}-6\epsilon}}\leq CM^{-3\epsilon}$$
,which in particular gives $\bar{F}^2(x,t)=0$ for $t\in[0,1-N^{-\frac{\epsilon}{2}}]$. Furthermore, using the  bounds for $\w(x,t)$ (see \eqref{w1bw3b}) we have
$$||g'(t)\w(x,t)||_{C^{\frac{1}{2}-6\epsilon}}\leq M^{\epsilon}\text{supp}_{t\in[1-N^{-\frac{\epsilon}{2}},1-N^{-\frac{\epsilon}{2}}+M^{-\epsilon}]}||\w(x,t)||_{C^{\frac{1}{2}-6\epsilon}}\leq CM^{-4\epsilon}.$$
This combined with \eqref{w25u35} gives us \eqref{cotas1}.
Next, using lemma \ref{ultimolemma} plus
$$||\p_{x_{1}}u_{1}(\w^{1})||\leq C_{0}N^{\epsilon}\leq \frac{4}{C_{0}}\ln(M)^2$$
gives us \eqref{cotas2} and finishes the proof.

\end{proof}

\begin{corollary}
    For any $\delta>0$, there exists a solution to the forced 3D-Euler equation in vorticity formulation $\w(x,t)$ 
 with force $F(x,t)$ fulfilling $\w(x,t)\in C^{\frac{1}{2}-\delta}$ for $t\in[0,1)$,
 $$\text{sup}_{t\in[0,1)}||F(x,t)||_{C^{\frac{1}{2}-\delta}}<\infty$$
 and
 $$\text{lim}_{t\rightarrow 1}||\w(x,t)||_{L^{\infty}}=\infty.$$
    
\end{corollary}

\begin{proof}
    First, given $\epsilon>0$, for some big $N$ we define 
    $\w^{0}(x,t)=(\w^{0}_{1}(x,t),\w^{0}_{2}(x,t),\w^{0}_{3}(x,t))$
with
$$\w^{0}_{1}(x,t)=0$$
$$\w^{0}_{2}(x,t)=-\int_{-\infty}^{x_{2}}\p_{x_{3}}\w^{0}_{3}(x_{1},s,x_{3})ds,$$
$$\w^{0}_{3}(x,t)=-M^{\epsilon}f(M^{\frac{1}{2}-\epsilon}x_{3})\sin(Mx_{1})f(M^{\frac{1}{2}+\epsilon}x_{1})\sin(Mx_{2})f(M^{\frac{1}{2}}x_{2}),$$
($f(z)\in C^{\infty}$, $0\leq f(z)\leq 1$, $f(z)=1$ if $|z|\leq \frac12$, $f(z)=0$ if $|z|\geq 1$), which is a solution to the forced incompressible 3D-Euler equation in vorticity formulation  with some force $F^{0}(x,t)\in C^{\infty}$ and $||F^{0}(x,t)||_{C^{1}}\leq C.$
Furthermore, we can apply lemma \ref{vaprox}  to $R(0,0,\w^{0}_{3})$ show that, for any $a>0$, if $N$ big, 
$$\frac{(1+a)C_{0}N^{\epsilon}}{2}\geq \p_{x_{1}}u_{1}(\w^{0}_{3})\geq \frac{(1-a)C_{0}N^{\epsilon}}{2}$$
which, combined with the bounds for $\w^{0}_{2}$ and \eqref{lnu} gives us, for $N$ big
$$C_{0}N^{\epsilon}\geq \p_{x_{1}}u_{1}(\w^{0})\geq \frac{C_{0}N^{\epsilon}}{4}.$$

With this $\w^{0}(x,t)$, we can apply theorem \ref{glue} with $\bar{\w}^{1}(x,t)=\w^{0}(x,t)$, $\bar{F}^{1}(x,t)=F^{0}(x,t)$ to obtain $\bar{\w}^{2}(x,t)$ such that 
$\w^{0}(x,t)+\bar{\w}^{2}(x,t)$
is a solution to the forced incompressible 3D-Euler equation in vorticity formulation  with force $F^{0}(x,t)+\bar{F}^{2}(x,t)$, fulfilling, if $M_{0}^{\frac{1}{2}}:=e^{\int_{1-N^{-\frac{\epsilon}{2}}}^{1}\p_{x_{1}}u_{1}(\w^0)(s,0)ds}$,
    \begin{equation*}
        ||\w^{0}(x,t)+\bar{\w}^2(x,t)||_{C^{2.5}},||u(\w^{0}(x,t)+\bar{\w}^2(x,t))||_{C^{3.5}}\leq M_{0}^{4},||\bar{F}^2(x,t)||_{C^{\frac{1}{2}-6\epsilon}}\leq M_{0}^{-\epsilon}
    \end{equation*}
and for $t\in[1-M_{0}^{-\frac{\epsilon}{2}},1]$
    \begin{equation*}
        C_{0}M_{0}^{\epsilon}\geq \p_{x_{2}}u_{2}(\w^{0}+\bar{\w}^2)\geq \frac{C_{0}}{4}M_{0}^{\epsilon}, ||\p_{x_{1}}u_{1}(\w^0+\bar{\w}^{2})||_{L^{\infty}}\leq \ln(M_{0})^{3},
    \end{equation*}
and, for  $t\in[0,1-N^{\frac{\epsilon}{2}}]$, $\bar{F}^2(x,t)=\bar{\w}^{2}(x,t)=0$.

We can define then $\w^{1}(x,t)=\w^{0}(x,t)+\bar{\w}^{2}(x,t)$ and iterate the process. More precisely, we apply theorem \ref{glue} with $\bar{\w}^1(x,t)=R^{-i}(\w^{i}(x,t))$, which satisfies the hypothesis of the theorem, to obtain a new $\bar{\w}^{2}(x,t)$ such that $R^{-i}(\w^{i} (x,t))+\bar{\w}^{2}(x,t)$ is a solution to the forced incompressible 3D-Euler equation in vorticity formulation  with force $R^{-i}(F^{i}(x,t))+\bar{F}^{2}(x,t)$, and, if we define $M_{i}^{\frac{1}{2}}:=e^{\int_{1-M_{i-1}^{-\frac{\epsilon}{2}}}^{1}\p_{x_{1}}u_{1}(\w^i)(s,0)ds}$, we have
    \begin{equation*}
        ||R^{-i}(\w^{i}(x,t))+\bar{\w}^2(x,t)||_{C^{2.5}},||u(R^{-i}(\w^{i}(x,t))+\bar{\w}^2(x,t))||_{C^{3.5}}\leq M_{i}^{4},||\bar{F}^2(x,t)||_{C^{\frac{1}{2}-6\epsilon}}\leq M_{i}^{-\epsilon}
    \end{equation*}
and for $t\in[1-M_{i}^{-\frac{\epsilon}{2}},1]$
    \begin{equation*}
        C_{0}M_{i}^{\epsilon}\geq \p_{x_{3}}u_{3}(R^{-i}(\w^{i}(x,t))+\bar{\w}^2)\geq \frac{C_{0}}{4}M_{i}^{\epsilon}, ||\p_{x_{1}}u_{1}(R^{-i}(\w^{i}(x,t))+\bar{\w}^{2})||_{L^{\infty}}\leq \ln(M_{i})^{3},
    \end{equation*}
and $\bar{\w}^2(x,t))=0$ if $t\in[0,1-M_{i-1}^{-\frac{\epsilon}{2}}]$.

Then, we define $\w^{i+1}(x,t)=\w^{i}(x,t)+R^{i}(\bar{\w}^{2}(x,t))$, which is a solution to the forced incompressible 3D-Euler equation in vorticity formulation with force $F^{i+1}(x,t)=F^{i}(x,t)+R^{i}(\bar{F}^{2}(x,t))$.
Finally, the solution that blows up will be
$$\w^{\infty}(x,t)=\text{lim}_{i\rightarrow\infty}\w^{i}(x,t).$$

We note that the limit trivially exists for all $t\in[0,1)$, since, for $t\in[0,1-M_{j}^{-\frac{\epsilon}{2}}]$, we have that, for $i_{1},i_{2}\geq j+1$
$$\w^{i_{1}}(x,t)=\w^{i_{2}}(x,t),\ F^{i_{1}}(x,t)=F^{i_{2}}(x,t)$$
and $\w^{i}(x,t)\in C^{2.5}$ for $t\in[0,1]$.

For the same reason, we have that for $t\in[0,1)$, $\w^{\infty}(x,t)$ is a solution to the forced incompressible 3D-Euler equation in vorticity formulation, since $\w^{i}(x,t)$ is a solution to the forced incompressible 3D-Euler equation in vorticity formulation for $t\in[0,1]$.
To obtain bounds for the force,  for any fixed $t_{0}\in[0,1)$ we have
$$||F^{\infty}(x,t_{0})||_{C^{\frac{1}{2}-6\epsilon}}\leq ||F_{0}(x,t)||_{C^{\frac{1}{2}-6\epsilon}}+\sum_{j=0}^{\infty}(M_{j})^{-\epsilon}\leq C$$
where we used that for $j$ big $M_{j}^{-\epsilon}<< j^{-2}$.

Finally, we only need to check that the solution blows up, but we know that, for $t_{i}=1-M_{i-1}^{-\frac{\epsilon}{2}}$

$$|u(\w^{\infty}(x,t_{i}))|_{C^1}=|u(\w^{i}(x,t_{i}))|_{C^1}\geq \frac{C_{0}}{4}M_{i-1}^{\epsilon},$$
$$||\w^{\infty}(x,t_{i})||_{C^{2.5}}=||\w^{i}(x,t_{i})||_{C^{2.5}}\leq M_{i-1}^{4}$$
and using that
$$ |u(\w^{\infty}(x,t_{i}))|_{C^1}\leq C||\w^{\infty}(x,t_{i})||_{L^{\infty}}\ln(10+||\w^{\infty}(x,t_{i})||_{C^{2.5}})$$
gives us
$$||\w^{\infty}(x,t_{i})||_{L^{\infty}}\geq \frac{CM_{i-1}^{\epsilon}}{\ln(10+M_{i-1}^{4})}$$
and since $M_{i-1}$ tends to infinity as $i$ tends to infinity this finishes the proof.
    
\end{proof}

\section*{Acknowledgements}
This work is supported in part by the Spanish Ministry of Science
and Innovation, through the “Severo Ochoa Programme for Centres of Excellence in R$\&$D (CEX2019-000904-S)” and 114703GB-100. We were also partially supported by the ERC Advanced Grant 788250. 

\bibliographystyle{alpha}

\end{document}